\documentclass[11pt, oneside]{amsart}

\usepackage{tabularx, hyperref}
\usepackage{amssymb} \usepackage{amsfonts} \usepackage{amsmath}
\usepackage{amsthm} \usepackage{epsfig, subfig}
\usepackage{ amscd, amsxtra, latexsym}
\usepackage[all]{xy}
\usepackage{caption}
\usepackage{enumerate}
\usepackage{color}

\newtheorem{lemma}{Lemma}[section]
\newtheorem{thm}[lemma]{Theorem}
\newtheorem{prop}[lemma]{Proposition}
\newtheorem{cor}[lemma]{Corollary}

\newtheorem*{prop*}{Proposition}
\newtheorem{prop_intro}{Proposition}
\newtheorem{thm_intro}[prop_intro]{Theorem}

\newtheorem{cor_intro}[prop_intro]{Corollary}
\theoremstyle{definition}
\newtheorem{defn}[lemma]{Definition}

\newtheorem{quest}[lemma]{Question}

\newtheorem{rem}[lemma]{Remark}

\theoremstyle{definition}

\newtheorem*{IS}{Informal Statement}

\definecolor{darkgreen}{cmyk}{1,0,1,.2}

\newcommand{\g} {\ensuremath {\gamma}}
\DeclareMathOperator{\res}{res}
\DeclareMathOperator{\diam}{diam}
\DeclareMathOperator{\tr}{tr}
\DeclareMathOperator{\alt}{alt}

\newcommand{\QCa}{{\rm QZ}_{\alt}}
\newcommand{\QC}{{\rm QZ}}
\newcommand{\In}{{\rm in}}
\newcommand{\out}{{\rm out}}

\newcommand{\Cay}{\ensuremath{{\rm Cay}}}

\renewcommand{\L}{\ensuremath {\Lambda}}

\renewcommand{\l}{\ensuremath {\lambda}}

\newcommand{\ov}{\ensuremath {\overline}}
\newcommand{\N}{\ensuremath {\mathbb{N}}}
\newcommand{\R} {\ensuremath {\mathbb{R}}}

\newcommand{\Z} {\ensuremath {\mathbb{Z}}}

\newcommand{\calC} {\ensuremath {\mathcal{C}}}

\newcommand{\calQ} {\ensuremath {\mathcal{Q}}}
\newcommand{\calP} {\ensuremath {\mathcal{P}}}
\newcommand{\calS} {\ensuremath {\mathcal{S}}}

\newcommand{\calH} {\ensuremath {\mathcal{H}}}

\newcommand{\calA} {\ensuremath {\mathcal{A}}}
\newcommand{\calB} {\ensuremath {\mathcal{B}}}

\newcommand{\calR} {\ensuremath {\mathcal{R}}}

\begin{document}

\title{Extending higher dimensional quasi-cocycles}

\author[]{R. Frigerio}
\address{Dipartimento di Matematica, Universit\`a di Pisa, Largo B. Pontecorvo 5, 56127 Pisa, Italy}
\email{frigerio@dm.unipi.it}

\author[]{M. B. Pozzetti}
\address{Department Mathematik, ETH Z\"urich, 
R\"amistrasse 101, CH-8092 Z\"urich, Switzerland}
\email{beatrice.pozzetti@math.ethz.ch}

\author[]{A. Sisto}
\address{Department Mathematik, ETH Z\"urich, 
R\"amistrasse 101, CH-8092 Z\"urich, Switzerland}
\email{sisto@math.ethz.ch}

\thanks{The authors would like to thank the organizers of the conference ``Geometric and analytic group theory'' (Ventotene 2013) where 
part of this work was carried out, as well as Michelle Bucher, Marc Burger, Alessandra Iozzi and Denis Osin for helpful discussions and clarifications. 
Beatrice Pozzetti was partially supported by SNF grant no. 200020\textunderscore144373.}

\keywords{Bounded cohomology, quasi-cocycles, hyperbolically embedded groups, coset projections}
\begin{abstract}
Let $G$ be a group admitting a non-elementary acylindrical action on a Gromov hyperbolic space
(for example, a non-elementary relatively hyperbolic group, or the mapping class group of a closed hyperbolic surface, or
${\rm Out}(F_n)$ for $n\geq 2$).
We prove that, in degree 3, the bounded cohomology
of $G$ with real coefficients is infinite-dimensional. Our proof is based on an extension to higher degrees
of a recent result by Hull and Osin. Namely, we prove that, if $H$ is a hyperbolically embedded subgroup of $G$ and $V$ is any $\mathbb{R}[G]$-module,
then any $n$-quasi cocycle on $H$ with values in $V$  
may be extended to $G$.
Also, we show that our
extensions detect the geometry of the embedding of hyperbolically embedded
subgroups, in a suitable sense.
\end{abstract}

\maketitle

Bounded cohomology of discrete groups is very hard to compute. For example,
as observed in~\cite{Monod}, there is not a single countable group $G$  
whose bounded cohomology (with trivial coefficients) is known in every degree, unless 
it is known to vanish in all positive degrees (this is the case, for example, of amenable groups). 
Even worse, no group $G$ is known for which the supremum of the degrees $n$ such that $H^n_b(G,\mathbb{R})\neq 0$ is positive and finite.

In degree 2, bounded cohomology has been extensively studied via the analysis of \emph{quasi-morphisms}
(see e.g.~\cite{Brooks, EpsteinFuji, Fujiwara1,BeFu-wpd, FujiTAMS} for the case of trivial coefficients, 
and~\cite{HullOsin,BBF2} for more general coefficient modules). In this paper we exploit \emph{quasi-cocycles},
which are the higher-dimensional analogue of quasi-morphisms, to prove non-vanishing results
for bounded cohomology in degree 3. To this aim, we prove that quasi-cocycles may be extended from a hyperbolically embedded family of subgroups
to the ambient group. We also discuss the geometric information carried by our extensions of quasi-cocycles, 
showing in particular that projections on hyperbolically embedded subgroups may be reconstructed from the extensions of suitably chosen
quasi-cocycles. 

\subsection*{Quasi-cocycles and bounded cohomology}
Let $G$ be a group, and $V$ be a normed $\R[G]$-space, i.e.~a normed real vector space endowed with an isometric left action of $G$.
We denote by $C^n(G,V)$ the set of homogeneous $n$-cochains on $G$ with values in $V$, 
and for every $\varphi\in C^n(G,V)$
we set
$$
\|\varphi\|_\infty=\sup \{\|\varphi (g_0,\ldots,g_n)\|_V\, |\, (g_0,\ldots,g_n)\in G^{n+1}\}\ \in \ [0,\infty]\ .
$$
We denote by $C^n_b(G,V)\subseteq C^n(G,V)$ the subspace
of bounded cochains, and by $C^n(G,V)^G$, $C^n_b(G,V)^G$ the subspaces of invariant (bounded) cochains 
(see Section~\ref{background:sec} for the precise definitions). The space of \emph{$n$-quasi-cocycles} is defined as follows:
$$
\QC^n(G,V)=\{\varphi\in C^n(G,V)\, |\, \delta^n\varphi\in C^{n+1}_b(G,V)\}\ .
$$
Roughly speaking, quasi-cocycles are those cochains whose differential is quasi-null. 
Just as in the case of quasi-morphisms, the \emph{defect} of a quasi-cocycle $\varphi\in \QC^n(G,V)$ is given by
$$
D(\varphi)=\|\delta^n\varphi\|_\infty\ .
$$
Any cochain which stays at bounded distance
from a genuine cocycle is a quasi-cocycle. The existence of $G$-invariant quasi-cocycles that are not at bounded distance from any 
$G$-invariant cocycle
is equivalent to the non-vanishing of the exact part $EH^{n+1}_b(G,V)$ of the bounded cohomology module $H^{n+1}_b(G,V)$ (see below for the
definition of $EH^{n+1}_b(G,V)$), so quasi-cocycles are a useful tool in the study of bounded cohomology. 

For technical reasons, it is convenient to consider the subspace of \emph{alternating} quasi-cocycles,
which is denoted by $\QCa^n(G,V)$.
Hull and Osin recently proved that if $G$ is a group and $\{H_\lambda\}_{\lambda\in\Lambda}$ is a hyperbolically embedded family
of subgroups of $G$, then alternating $1$-quasi-cocycles on the $H_\lambda$'s may be extended to $G$~\cite{HullOsin}. 
In this paper we extend Hull and Osin's result to higher dimensions:

\begin{thm_intro}\label{main:thm}
 Let $G$ be a group, let $\{H_\lambda\}_{\lambda\in\Lambda}$ be a  hyperbolically embedded  family of  subgroups
 of $G$, and let $V$ be a normed $\R[G]$-module. 
For every $n\geq 1$, there exists a linear map
 $$
 \Theta^n\colon \bigoplus_{\lambda\in\Lambda} \QCa^n(H_\lambda,V)^{H_\lambda}\to \QCa^n(G,V)^G 
$$
such that, for every $\varphi=(\varphi_\lambda)_{\lambda\in\Lambda}\in \bigoplus_{\lambda\in\Lambda} \QCa^n(H_\lambda,V)^{H_\lambda}$ and for every 
$\lambda\in\Lambda$, we have 
$$
\sup_{\overline{h}\in H_\lambda^{n+1}} \|\Theta^n(\varphi)(\overline{h})-\varphi_\lambda(\overline{h})\|_V <\infty\ .
$$
\end{thm_intro}
 
 We refer the reader to Theorem~\ref{mainc} for a more general statement.

 A natural question is whether 
 possibly non-alternating quasi-cocycles could also be quasi-extended from the $H_\lambda$'s to $G$. This is always true if $n=1$ since 
 1-quasi-cocycles are at bounded distance from alternating ones
 (see Remark~\ref{nonalternating:rem} for a brief discussion of this issue in higher degrees). However,
 it seems unlikely that our construction could be adapted to deal with the general case.

 \bigskip

Our proof of Theorem \ref{main:thm} is based on the construction, which we carry out in Section~\ref{trace:sec}, of the trace of a simplex on a coset, 
that we think of as a  projection of an $(n+1)$-tuple in $G^{n+1}$ to a given coset of a hyperbolically embedded subgroup.
Since the projection on a coset of a hyperbolically embedded subgroup is a multi-valued function, 
the trace of a simplex is not a single simplex, but an average of simplices. In order to maximize the number of cancellations
between traces of simplices
and reduce the technical effort in the proof of the main theorem,
we chose to work in the coned off graph $\widehat{G}$, that is obtained from a Cayley graph of $G$ by adding an extra point for any coset. 
The metric properties of the coned graph $\widehat{G}$ 
allow us to prove that, for $n>1$, given any $n$-simplex there is a set of at most $n(n+1)$ exceptional cosets
such that
the diameter of the trace of the simplex on any other coset is smaller than an universal constant.
Our results here are very similar to analogous results proved by Hull and Osin in~\cite{HullOsin},
and in fact our arguments were inspired by theirs (even if Hull and Osin's constructions take place in a slightly different
context). 
Perhaps, it is worth mentioning that the exceptional cosets associated to a simplex also generalize the barycenter of the simplex as defined
in \cite{BBFIPP}, where the case of amalgamated products is analyzed. 
Indeed, if $G=H\ast K$, then the family $\{H,K\}$ is hyperbolically embedded in $G$,
and our construction provides a ``quasification'' of the strategy described in~\cite{BBFIPP}.

\subsection*{Applications to bounded cohomology}
For any group $G$, any normed $\R[G]$-module $V$ and every $n\geq 0$, the inclusion of bounded cochains into ordinary cochains induces the \emph{comparison map} 
$c^n\colon H^n_b(G,V)\to H^n(G,V)$. The kernel of $c^n$ is the set of bounded cohomology classes whose representatives are
exact, and it is denoted by $EH^n_b(G,V)$. If $K$ is a subgroup of $G$, then the restriction of cochains on $G$
to cochains on $K$ induces the map $\res^\bullet\colon EH^\bullet_b(G,V)\to EH^\bullet_b(K,V)$.
Building on Theorem~\ref{main:thm}, in Section~\ref{nonvanishing:sec} we prove the following result (see Proposition~\ref{boundedcoc:prop} for a slightly more general
statement):

\begin{cor_intro}\label{boundedco:cor}
Let $G$ be a group, let $\{H_\lambda\}_{\lambda\in\Lambda}$ be a  hyperbolically embedded family of subgroups
of $G$, and let $V$ be a normed $\R[G]$-module. 
Fix $n\geq 2$, and denote by  $\res^n_\lambda\colon H^n_b(G,V)\to H^n_b(H_\lambda,V)$ the restriction map. 
For every element $(\alpha_\lambda)_{\lambda\in\Lambda}\in \bigoplus_{\lambda\in\Lambda}  EH^n_b(H_\lambda,V)$, there exists $\alpha\in EH^n_b(G,V)$ 
such that 
$$
\res^n_\lambda(\alpha)=
\alpha_\lambda\quad {\rm for\ every}\ \lambda\in\Lambda
\ .$$
\end{cor_intro}

The norm $\|\cdot\|_\infty$ on $C^n_b(G,V)$ induces a seminorm on $H^n_b(G,V)$ that is usually referred to as \emph{Gromov seminorm}.
Let us now denote by $N^n_b(G,V)$ the subspace of $H^n_b(G,V)$ given by elements with vanishing seminorm, and let us set
$\overline{H}^n_b(G,V)=H^n_b(G,V)/N^n_b(G,V)$, so  $\overline{H}^n_b(G,V)$ is a Banach space.
Following~\cite{Os-acyl}, we say that a group $G$ is \emph{acylindrically hyperbolic} if it admits an acylindrical action on
a Gromov hyperbolic space. It is shown in~\cite{Os-acyl} that being acylindrically hyperbolic is equivalent to containing a proper infinite hyperbolically embedded subgroup.
Building on results from~\cite{DGO}, from Corollary~\ref{boundedco:cor} we deduce the following:

\begin{cor_intro}\label{boundedco2:cor}
 Let $G$ be an acylindrically hyperbolic group. Then the dimension of both $\overline{H}^3_b(G,\R)$ 
 and $EH^3_b(G,\R)$
 is equal to the cardinality of the continuum. Therefore, the same is true also for $H^3_b(G,\R)$.
\end{cor_intro}

The class of acylindrically hyperbolic groups includes many examples of interest:
non-elementary hyperbolic and relatively hyperbolic groups~\cite{DGO}, the mapping class group
of the $p$-punctured closed orientable surface of genus $g$, provided that $3g+p\geq 6$~\cite[Theorem 2.18]{DGO}, 
${\rm Out}(F_n)$ for $n\geq 2$~\cite[Theorem 2.20]{DGO},
groups acting geometrically on a proper CAT(0)
space with a rank one isometry~\cite{Si-contr} and \cite[Theorem 2.22]{DGO},
and fundamental groups of several graphs of groups~\cite{MO}.

\subsection*{Further results}

One may wonder whether Corollary~\ref{boundedco:cor} holds with bounded cohomology instead of exact bounded cohomology. 
In fact, the map $\Theta^\bullet$ of Theorem~\ref{main:thm} extends to a map between alternating cochains sending bounded cochains to bounded cochains
(see Theorem~\ref{maincc}),
so one may wonder whether that $\Theta^\bullet$ could be used to extend possibly non-exact bounded coclasses. However, in general our map $\Theta^\bullet$
does not carry cocycles to cocycles, but only to quasi-cocycles, so it is not a chain map. In fact, 
in Section~\ref{counterexample:sec} we prove the following:

\begin{prop_intro}\label{1contro:prop}
 For every $n\geq 2$, there exists a pair $(G,H)$ such that $G$ is relatively hyperbolic with respect to $H$
(in particular, $H$ is hyperbolically embedded in $G$), and the restriction $H^n_b(G,\R)\to H^n_b(H,\R)$
is not surjective.
 \end{prop_intro}

Even worse, $\Theta^\bullet$ does not induce a well-defined map on exact bounded cohomology in general (see Proposition~\ref{secondocontro} for an explicit example). 
In order to obtain a positive result in this direction, we need to make some further assumptions
on the ordinary cohomology of the subgroups $H_\l$
(see Proposition~\ref{inverse2:thm}).


 However, the fact that $\Theta^\bullet$ does not induce a well-defined map on bounded cohomology may be exploited to
 prove non-vanishing results. Namely, it may happen that a genuine (unbounded) real $n$-cocycle on a 
 hyperbolically embedded subgroup $H$ of $G$ 
 may be extended to a quasi-cocycle on $G$ whose differential defines a non-trivial 
class in $H^{n+1}_b(G,\R)$. For example, we can prove the following result (see Corollary~\ref{nonv:cor}):

\begin{prop_intro}\label{prop:1}
Let $H$ be an amenable hyperbolically embedded subgroup of the group $G$, let $n\geq 1$, and suppose that the inclusion $H\to G$
induces a non-injective map $H_n(H,\R)\to H_n(G,\R)$. Then $H^{n+1}_b(G,\R)\neq 0$.
\end{prop_intro}

As a consequence of (a variation of) Proposition \ref{prop:1}, and building on a construction 
by McReynolds, Reid and Stover \cite{MRS-collisions}, in Proposition \ref{interm} we show that, for every $n\geq 3$ and
$2\leq k\leq n$, 
there exist infinitely many commensurability classes of cusped hyperbolic $n$-manifolds $M$
such that $H^k_b(M,\mathbb{R})\neq 0$.

\bigskip

A natural question is whether, given $n$, it is possible to find a hyperbolically embedded finite family of subgroups $\{H_\l\}_{\l\in\L}$ of a 
group $G$ so that the direct sum of 
the restriction map $\oplus\res^n_\lambda\colon EH^n_b(G,V)\to \bigoplus EH^n_b(H_\lambda,V)$
is an isomorphism (this map is surjective by Corollary \ref{boundedco:cor}). 
In dimension 3 this is never the case, due to the following:

\begin{prop_intro}\label{infinitekernel:prop}
Let $G$ be a finitely generated group, let $\{H_\lambda\}_{\lambda\in\Lambda}$ be a finite hyperbolically embedded family of subgroups
of $G$, and let $V$ be a normed $\R[G]$-module. 
Then the kernel of the restriction map $\oplus\res^3_\lambda\colon EH^3_b(G,V)\to \bigoplus EH^3_b(H_\lambda,V)$ is infinite-dimensional.
\end{prop_intro}

We obtain Proposition \ref{infinitekernel:prop} as a consequence of a refinement of a result of Dahmani, Guirardel and Osin \cite[Theorem 6.14]{DGO} 
that might be of independent interest.  
Recall that a family $\{ H_\l\}_{\l\in\L}$ of subgroups of $G$ is \emph{non-degenerate} if there is some $\l$ so that $H_\l$ is a proper, infinite subgroup of $G$. 
If $\{H_\lambda\}_{\lambda\in\Lambda}$ is a hyperbolically embedded family of subgroups
of $G$, then each $H_\lambda$ is hyperbolically embedded in $G$ (see the first sentence of Remark~\ref{finitelygen:rem}).
In particular, a consequence of  \cite[Theorem 6.14]{DGO} is that, if a group 
$G$ contains a non-degenerate hyperbolically embedded family of subgroups, then $G$ contains a maximal finite normal subgroup, which will be denoted by $K(G)$.

\begin{thm_intro}\label{thm:DGO+e}
 Let $X$ be a (possibly infinite) generating system of the group $G$ and let the non-degenerate family of subgroups $\{H_\l\}_{\l\in\L}$ be 
 hyperbolically embedded in $(G,X)$. 
 Then for each $n\geq 1$ there exists a copy $F$ of the free group on $n$ generators inside $G$ so that $\{H_\l\}_{\l\in\L}\cup\{F\times K(G)\}$ is hyperbolically embedded in $(G,X)$.
\end{thm_intro}

Since the proof of Theorem \ref{thm:DGO+e} uses techniques somehow different from the rest of the paper, and is heavily based on results of \cite{DGO}, whereas the rest of the paper 
is almost self-contained, we decided to include the proof of Theorem \ref{thm:DGO+e} in an appendix, rather than in the main body of the paper.

 \subsection*{Quasi-cocycles and projections}
Any hyperbolically embedded family $\{H_\l\}_{\l\in\L}$ of subgroups of a group $G$ comes along with a family of $G$-equivariant projections $\pi_B:G\to B$ for every coset $B$ of a subgroup $H_\lambda$, satisfying certain axioms first introduced by Bestvina, Bromberg and Fujiwara in \cite{BBF}. It can be shown that 
the family of projections itself captures the fact that the family $\{H_\l\}_{\l\in\L}$ is hyperbolically embedded in $G$. Theorem \ref{BBF:thm}, reported below, makes this statement precise, combining results in the literature. The BBF axioms are defined in Section \ref{sec:proj}.

\begin{thm_intro}\label{BBF:axioms}
  Let $\{H_\lambda\}_{\lambda\in\Lambda}$ be a finite family of finitely generated subgroups of the
finitely generated group $G$, and let $\mathcal{B}$ be the set of the (labelled) cosets of the $H_\l$'s.
\begin{enumerate}
  \item Suppose that it is possible to assign, for each pair of cosets
$B_1,B_2\in\mathcal{B}$, a subset $\widetilde\pi_{B_1}(B_2)\subseteq
B_1$ in an equivariant way (i.e. in such a way that
$\widetilde\pi_{gB_1}(gB_2)=g\widetilde\pi_{B_1}(B_2)$) and so that the BBF axioms are
satisfied. Then $\{H_\l\}_{\l\in\L}$ is hyperbolically embedded in $G$.
  \item Suppose $\{H_\l\}_{\l\in\L}$ is hyperbolically embedded in $(G,X)$. Then the family of projections
  $\{\pi_B\}_{B\in\calB}$ as in Definition \ref{projdefn} satisfies the BBF axioms.
\end{enumerate}
\end{thm_intro}

Our extension of quasi-cocycles is only based on the good properties of the family of projections $\pi_B$ and already contains all the information necessary to reconstruct the projections themselves. In fact, in Section \ref{sec:reconsproj} we make the following statement precise:

\begin{IS}\label{reconstruct:thm}
Let $H$ be a finitely generated group. Then,
there exist a coefficient module $V$ and a cocycle $c\in C^2_{{\rm alt}}(H,V)$ such that the following holds.
 Whenever $H$ is hyperbolically embedded in $G$, the projections on the cosets of $H$ may be
 recovered from the extension $\Theta^2(c)$, which, therefore, 
detects the geometry of the embedding of $H$ 
 in $G$.
\end{IS}

So, by exploiting projections (and Theorem~\ref{BBF:axioms}) we are able to ``close the circle'' and get back 
from our cocycle extensions to the fact $H$ is hyperbolically embedded in $G$. We hope that in the future this will lead to a complete characterization of hyperbolically embedded subgroups in terms of bounded cohomology.

We emphasize that our argument does rely on $c$ being a cocycle of dimension greater than 1, 
and the authors are not aware of ways to reconstruct projections using quasi-morphisms.

\section{Basic facts about bounded cohomology}\label{background:sec}

Let us recall some basic definitions about bounded cohomology of groups.
Let $G$ be a group, and $V$ be a normed $\R[G]$-space.
The set of $n$-cochains on $G$ with values in $V$ is given by
$$
C^n(G,V)=\left\{ \varphi\colon G^{n+1}\to V\right\}\ .
$$
The vector space $C^n(G,V)$ is endowed with a left action of $G$ defined by
$(g\cdot \varphi) (g_0,\ldots,g_n)=g\cdot (\varphi(g^{-1}g_0,\ldots,g^{-1}g_n))$. 
We have defined in the introduction the submodule
 $C^n_b(G,V)\subseteq C^n(G,V)$ of bounded cochains. The action of $G$ on $C^n(G,V)$ preserves $C^n_b(G,V)$,
 so $C^n_b(G,V)$  is a normed $\R[G]$-module. The differential 
$$
\begin{array}{c}
\delta^n\colon C^n(G,V)\to C^{n+1}(G,V)\, ,\\ \delta^n\varphi (g_0,\ldots,g_{n+1})=\sum_{j=0}^{n+1} (-1)^j\varphi(g_0,\ldots,\widehat{g}_j,\ldots,g_{n+1})
\end{array}
$$
restricts to a map $C^n_b(G,V)\to C^{n+1}_b(G,V)$, which will still be denoted by $\delta^n$. 
If $W$ is a (normed) $\R[G]$-module, then we denote by $W^G$ the subspace of $G$-invariant elements of $W$. 
The differential $\delta^n$ sends invariant cochains to invariant cochains, thus endowing
$C^\bullet(G,V)^G$ and $C^\bullet_b(G,V)^G$ with the structure of chain complexes. 
The cohomology
(resp.~bounded cohomology)
of $G$ with coefficients in $V$ is the cohomology of the complex
$C^\bullet(G,V)^G$ (resp.~$C^\bullet_b(G,V)^G$).

Let us denote by $\mathfrak{S}_{n+1}$ the group of permutations of $\{0,\ldots,n\}$.
A cochain $\varphi\in C^n(G,V)$ is \emph{alternating} if
$$
\varphi(g_{\sigma(0)},\ldots,g_{\sigma(n)})={\rm sgn}(\sigma)\cdot \varphi (g_0,\ldots,g_n)
$$
for every $\sigma\in\mathfrak{S}_{n+1}$. 
Both the differential and the $G$-action preserve alternating cochains, that hence give a subcomplex
$C^\bullet_{\rm {alt}}(G,V)^G$ of $C^\bullet(G,V)^G$, respectively $C^\bullet_{b,{\alt}}(G,V)^G$ of $C_b^\bullet(G,V)^G$.
The space of alternating quasi-cocycles $\QCa^n(G,V)$ is just the intersection of $\QC^n(G,V)$ with
$C^n_{\rm alt}(G,V)$.

For every $n\geq 0$ we denote by $C_n(G)$ the real vector space 
with basis $G^{n+1}$. Elements of $G^{n+1}$ are called \emph{$n$-simplices}, and we say
that an $n$-simplex $\overline{g}=(g_0,\ldots,g_n)$ is supported in a subset $S\subseteq G$ if
all its vertices lie in $S$, i.e.~if $g_j\in S$ for every $j=0,\ldots,n$. The subspace of $C_n(G)$
generated  by simplices supported in $S$ is denoted by $C_n(S)$.
We also put on $C_n(G)$ the $\ell^1$-norm defined by
$$
\left\|\, \sum_{\overline{g}\in G^{n+1}} a_{\overline{g}} \overline{g}\, \right\|_1=\sum_{\overline{g}\in G^{n+1}} |a_{\overline{g}}|\ .
$$
If $\overline{g}=(g_0,\ldots,g_n)\in C_n(G)$, we denote by $\partial_j \overline{g}=(g_0,\ldots,\widehat{g_j},\ldots,g_n)\in
C_{n-1}(G)$ the $j$-th face of $\overline{g}$, and we set $\partial \overline{g}=\sum_{j=0}^n (-1)^j \partial_j\overline{g}$.

\subsection*{Degenerate chains}
If $S\subseteq G$ is any subset, then we may define an alternating linear operator
$
\alt_n\colon C_n(S)\to C_n(S)
$
by setting, for every $\overline{s}=(s_0,\ldots,s_n)\in S^{n+1}$,
$$
\alt_n(\overline{s})=\frac{1}{(n+1)!} \sum_{\sigma\in\mathfrak{S}_{n+1}} {\rm sgn}(\sigma)
(s_{\sigma(0)},\ldots,s_{\sigma(n)})\ .
$$
A chain $c\in C_n(S)$ is \emph{degenerate} if $\alt_n(c)=0$. If $K$ is a group and $W$ is an $\mathbb{R}[K]$-module, 
then it is immediate to check that
a cochain $\varphi\in C^n(K,W)$ is alternating if and only if it vanishes on degenerate chains in $C_n(K)$.
If $\varphi\in C^n(S,V)$ is any cochain, then we may alternate it by setting
$$
\alt^n(\varphi)(\overline{s})=\varphi(\alt_n(\overline{s}))
$$
for every $\overline{s}\in S^{n+1}$.

In every degree, the $G$-equivariant
chain map $\alt^n\colon C^n(G,V)\to C^n (G,V)$ provides a linear projection onto the subcomplex
of alternating cochains, and $\alt^\bullet$ is $G$-equivariantly homotopic to the identity
(see e.g.~\cite[Appendix B]{Fuji_Man}). 

Moreover, $\alt^\bullet$ 
restricts to a $G$-equivariant chain map $\alt_b^\bullet\colon C^\bullet_b(G,V)\to C^\bullet_{b,\alt} (G,V)$, and for every
$n\in\mathbb{N}$ the map $\alt_b^n$ provides a norm non-increasing
projection onto $C^n_{b,\alt}(G,V)$. The homotopy between $\alt^\bullet$ and the identity of $C^\bullet(G,V)$ may be chosen
in such a way that it
restricts to a homotopy between $\alt_b^\bullet$ and the identity of $C_b^\bullet(G,V)$, which is bounded in every degree.
As a consequence, the bounded cohomology of $G$ with coefficients in $V$ may be computed as the cohomology of the complex
$C^\bullet_{b,\alt}(G,V)$.

Observe that, if $\varphi\in C^n(G,V)$ is a quasi-cocycle, then
$$
\|\delta^n\alt^n(\varphi)\|_\infty=\|\alt^n(\delta^n(\varphi))\|_\infty\leq \|\delta^n(\varphi)\|_\infty<+\infty\ ,
$$
so $\alt^n$ projects $\QC^n(G,V)$ onto $\QCa^n(G,V)$.

\begin{rem}\label{nonalternating:rem}
It is well-known that a quasi-cocycle $\varphi\in \QC^1(G,V)$ is at bounded distance from the alternating quasi-cocycle
$\alt^1(\varphi)\in\QCa^1(G,V)$. In fact, let $T^\bullet$ be a chain homotopy  between $\alt^\bullet$ and the identity of $C^\bullet(G,V)$
which preserves boundedness of cochains.
Then $T^1(\varphi)$, being a $G$-equivariant $0$-cochain, is bounded (if the action of $G$ is trivial, then it is even constant), so 
$$
\alt^1(\varphi)-\varphi=T^2(\delta^1 \varphi)- \delta^0(T^1 \varphi)
$$
is itself bounded. On the contrary, if $\varphi\in \QC^n(G,V)$, $n\geq 2$, then the cochain $T^{n+1}(\delta^n\varphi)$
is still bounded, while in general $\delta^{n-1}(T^n\varphi)$ 
(whence $\alt^n(\varphi)-\varphi$) can be unbounded.

Let us consider for example the group $\Z^2=\langle a,b\rangle$ and the 1-cocyles $\alpha,\beta\in Z^1(\Z^2,\R)\cong {\rm Hom}(\Z^2,\R)$ 
corresponding to the homomorphisms $\alpha',\beta'$ 
such that $\alpha'(a)=\beta'(b)=1$, $\alpha'(b)=\beta'(a)=0$. It is readily seen that, for every $n\in\Z$,
$(\alpha\cup \beta)(1,b^n,a^n)=0$, while $(\alpha\cup\beta)(a^n,1,b^n)=-n^2$. This implies
that the $2$-cocycle $\alpha\cup\beta$ does not lie at bounded distance from any
alternating cochain in $C^2_{\alt} (\Z^2,\R)$.
\end{rem}

\section{Projections and hyperbolic embeddings}\label{sec:proj}
Let $G$ be a group, and let us fix 
a family $\{H_\l\}_{\l\in\L}$  of subgroups of  $G$. A (possibly infinite) subset $X\subseteq G$
is a \emph{relative generating set} if $X\cup \bigcup_{\lambda\in\Lambda} H_\lambda$ generates $G$.

\begin{defn}[\cite{DGO}]\label{defhypemb}
Let $X$ be a relative generating set for $G$ and let  $\mathcal H$ denote the disjoint union $\mathcal H=\sqcup_{\l\in\L} H_\l\backslash\{e\}$.
We denote by ${\rm Cay}(G, X\sqcup \mathcal H)$ the Cayley graph of $G$ with respect to the alphabet $X\sqcup \mathcal{H}$. Notice that some letters
in $X\sqcup \mathcal{H}$ may represent the same element of $G$, in which case ${\rm Cay}(G, X\sqcup \mathcal H)$  has multiple edges corresponding
to these letters. We label each edge of ${\rm Cay}(G, X\sqcup \mathcal H)$ by the corresponding letter in $X\sqcup \mathcal{H}$.
 
For every $\l\in\Lambda$, we define the \emph{relative metric} $d_\l:H_\l\times H_\l\to [0,\infty]$ by letting 
$ d_\l(g,h)$ be the length of the shortest path in ${\rm Cay}(G, X\sqcup \mathcal H)$ 
that connects $g$ to $h$ and has no edge that connects vertices of $H_\l$ and is labelled by an element of $H_\l\backslash \{1\}$.

The family $\{H_\l\}_{\l\in\L}$ is \emph{hyperbolically embedded} in $(G,X)$ if the 
 Cayley graph ${\rm Cay}(G, X\sqcup \mathcal H)$ is hyperbolic and, for every $\l\in\L$, the metric space $(H_\lambda,d_\l)$ is locally finite. 
 
 In general, one says that $\{H_\l\}_{\l\in\L}$  is hyperbolically embedded
 in $G$ if it is hyperbolically embedded in 
 $(G,X)$ for some relative generating set $X\subseteq G$.
 In this case we write $\{H_\l\}_{\l\in\L}\hookrightarrow_h G$ or $\{H_\l\}_{\l\in\L}\hookrightarrow_h (G,X)$ when we want to emphasize the choice of $X$.
 
\end{defn}

Let us fix once and for all a subset $X\subseteq G$ so that $\{H_\lambda\}_{\lambda\in\Lambda}$ is hyperbolically embedded in $(G,X)$. 
Throughout the whole paper, any coset will be understood to be a \emph{left} coset.
We denote by $\calB$ the set of cosets of the subgroups $H_\lambda$, $\lambda\in\Lambda$. More precisely, we let $\calB$ be the \emph{disjoint} union
if the ${B}_\l$'s, where
$\mathcal{B}_\l$ is the set of cosets of $H_\l$ for every $\l\in\L$. We label every element of $\mathcal{B}_\l$ by the index $\l\in\L$. Notice that,
if there are repetitions among the $H_\lambda$'s, then some cosets appear with repetitions (but with distinct labels) in $\calB$. 

\begin{rem}
Let $\{H_\l\}_{\l\in\L}\hookrightarrow_h G$ and fix an index $\l\in\L$. Using the fact that the relative metric $d_\l$ is locally finite it is easy to prove
that, if $H_\l=H_{\l'}$ for some $\l'\neq \l$, then $H_\l$ is finite and the set of indices $\l''\in\L$ such that $H_{\l''}=H_\l$ is finite. Therefore,
at least for what concerns the main results of our paper, we could safely restrict our attention to the case when $H_\l\neq H_{\l'}$ for every $\l'\neq \l$.
\end{rem}

We now define the object we will work with  throughout the paper.

\begin{defn}\label{coned}
 We denote by $(\widehat{G},\widehat{d})$ the metric graph obtained by adding to ${\rm Cay}(G,X)$ a vertex $c(B)$ for each 
$B\in\calB$ and edges $[c(B),h]$ of length $1/4$ for every $c(B)$ and $h\in B$. 
\end{defn}

Our $\widehat{G}$ is very similar to Farb's coned-off graph \cite{Fa-relhyp}, but using $\widehat{G}$ rather than the coned-off graph or ${\rm Cay}(G,X\cup\mathcal H)$ will allow us to streamline a few arguments. 
Hopefully, $\widehat{G}$ will turn out to be more convenient in other contexts as well.

If a geodesic $\gamma$  of $\widehat{G}$ contains the vertex $c(B)$, then we denote by
 $\In_\gamma(B)$ and $\out_\gamma(B)$ respectively the last point of $\gamma\cap B$
 preceding $c(B)$  and the first point of $\gamma\cap B$
 following $c(B)$ along $\gamma$. If $\gamma$ starts (resp.~ends) at $c(B)$, then
 $\out_\gamma (B)$ (resp.~$\In_\gamma (B)$) is not defined.
 
 \begin{rem}\label{aaarem}
  Suppose that  the geodesic $\gamma$  of $\widehat{G}$ intersects
 the coset $B$ in at least two points $p,q$. Then  
 $\gamma$ contains $c(B)$, and 
 $\gamma\cap B=\{\In_\gamma(B),\out_\gamma(B)\}=\{p,q\}$. 
 \end{rem}

If $B\in\mathcal{B}$ is labelled by $\l\in\L$, we endow $B$ with the relative metric $d_B$
 obtained by translating ${d}_{\lambda}$. 
 If $S\subseteq B$ is a subset of some coset $B\in\calB$, then we denote by $\diam_{B}(S)$ the diameter of $S$
with respect to $d_B$. 

\begin{rem}\label{finitelygen:rem}
Let $\l_0\in \L$ be fixed. 
Since $\{H_\l\}_{\l\in\L}\hookrightarrow_h (G,X)$, 
we have $H_{\l_0}\hookrightarrow_h(G,X\cup \calH')$, where $\calH'=\bigcup_{\omega\neq \l_0} H_\omega$.
Therefore, \cite[Corollary 4.32]{DGO} implies that, if $G$ is finitely generated, then each $H_\l$ is finitely generated.
If, in addition, the family $\{H_\l\}_{\l\in\L}$ is finite, then
by~\cite[Corollary 4.27]{DGO}  we can
add to $X$ the union of finite generating sets of
the $H_\l$'s without altering the fact that
$\{H_\l\}_{\l\in\L}\hookrightarrow_h (G,X)$. Then~\cite[Lemma 4.11-(b)]{DGO} implies that 
the relative metric $d_B$ is bi-Lipschitz equivalent to a word metric on $B$.
 
 Hence, for the purposes of our paper, we could replace $d_B$ with a more familiar word metric
 whenever we deal with finite families of hyperbolically embedded subgroups of a finitely generated group. 
\end{rem}

\subsection*{Projections on cosets}
Projections, as defined below, will play a crucial role in this paper.

\begin{defn}\label{projdefn}
For every coset $B\in\calB$ and every
vertex $x$ of $\widehat{G}$, the projection of
$x$ onto $B$ is the set
$$
\pi_{B}(x)=\{p\in B\, |\, \widehat{d}(x,p)=\widehat{d}(x,B)\}\ .
$$
If $S\subseteq \widehat{G}$ is any subset, then
we set $\pi_{B}(S)=\bigcup_{x\in S} \pi_{B}(x)$.
\end{defn}

\begin{rem}
 Of course $\pi_B(c(B))=B$, while 
if $x\neq c(B)$ we have
 $$
 \pi_{B}(x)=\{p\in B\, |\, p=\In_{\gamma}(B)\, ,\ \gamma\ {\rm geodesic\ joining\ } x\ {\rm to}\ c(B)\}\ .
 $$
\end{rem}

An important result about projections is described in Lemma~\ref{farproj}, 
which says that if two points project far away on a coset $B$ then any geodesic connecting them contains $c(B)$.
Similar properties are also true for other notions of projections in a relatively hyperbolic space, as discussed in \cite{Si-proj}.
Also, the following lemma has strong connections with the bounded coset penetration property for Farb's coned-off graph \cite{Fa-relhyp}.

\begin{lemma}\label{farproj}
 There exists $D\geq 1$ with the following property. For $x,y\in \widehat{G}$ and a coset $B$, if $\diam_B(\pi_{B}(x)\cup \pi_{B}(y))\geq D$ then all geodesics from $x$ to $y$ contain $c(B)$.
\end{lemma}

\begin{proof}
 Let as above $\mathcal H=\bigsqcup H_\lambda\backslash\{1\}$ and set $\Gamma={\rm Cay}(G,X\sqcup \mathcal H)$. Let $\widehat{\gamma}_1$, $\widehat{\gamma}_2$ be geodesics in $\widehat{G}$ from $x$ to any $p\in\pi_B(x)$ and from $y$ to any $q\in\pi_B(y)$ respectively. Notice that $\widehat{\gamma}_1\cap B$ and $\widehat{\gamma}_2\cap B$ 
 each consists of a single point. We can form paths $\gamma_i$ in $\Gamma$ replacing all subpaths of $\widehat{\gamma}_i$ consisting of two edges intersecting at $c(B')$, for some coset $B'$, with an edge in $\Gamma$ 
 (and possibly removing the first edge of $\widehat{\gamma}_i$ if $x$ and/or $y$ are in $\widehat{G}$ but not in $G$). Consider now a geodesic $\widehat{\gamma}$ from $x$ to $y$, 
 and construct a path $\gamma$ in $\Gamma$ similarly (and possibly add an edge at the beginning/end of $\gamma$ to make sure that the endpoints of $\gamma$ coincide with the 
 starting points of $\gamma_1,\gamma_2$). Finally, if $\l$ is the label of $B$, let $e$ be the edge in $\Gamma$ labelled by 
an element of $H_\lambda$ connecting the endpoint of $\gamma_1$ to the endpoint of $\gamma_2$.
 
 It is not hard to see that (the unit speed parametrizations of) $\gamma,\gamma_i$ are, say, (2,2)-quasi-geodesics in $\Gamma$. For example, one can argue as follows. Given a geodesic $\alpha$ in $\Gamma$, we can replace each edge of $\alpha$ labelled by a letter from $\mathcal H$ by a path of length $1/2$ in $\widehat{G}$. This implies that $d_{\widehat{G}}(g,h)\leq d_\Gamma(g,h)$ for each $g,h\in G$. Now, whenever $g,h$ are on, say, $\gamma$, and $\gamma|_{g,h},\widehat{\gamma}|_{g,h}$ denote the subpaths of $\gamma,\widehat{\gamma}$ with endpoints $g,h$, we have
 $$l(\gamma|_{g,h})\leq 2 l(\widehat{\gamma}|_{g,h})= 2d_{\widehat{G}}(g,h)\leq 2d_\Gamma(g,h),$$
 which easily implies that $\gamma$ is a (2,2)-quasi-geodesic (we used in the equality that $\widehat{\gamma}$ is a geodesic). We proved that it is a (2,2)-quasi-geodesic rather than a (2,0)-quasi-geodesic because we showed the above inequality for $g,h\in G$ only. In fact, this estimate can be improved but we will not need to.
 
 The paths $\gamma_1,e,\gamma_2,\gamma$ form a (2,2)-quasi-geodesic quadrangle in $\Gamma$, which is a hyperbolic metric space. Hence, there exists $C$ depending on the hyperbolicity constant only so that any point on one side of the quadrangle is contained in the $C$-neighborhood of the union of the other three sides. Assume now that $\widehat{\gamma}$ does not contain $c(B)$ and hence that $\gamma$ does not contain any edge connecting points in $B$. Under this assumption we now construct a cycle $c$ whose length is bounded in terms of the hyperbolicity constant of $\Gamma$ and so that the only edge contained in $c$ that connects points in $B$ is $e$. 
 Such cycle is either a quadrangle, a pentagon or a hexagon formed by $e$, subpaths of $\gamma_1, \gamma_2$, possibly a subpath of $\gamma$ and one or two paths of length bounded in terms of the hyperbolicity constant of $\Gamma$. The idea is illustrated in Figure \ref{cutcycle}.
 
 \begin{figure}[h]
 \includegraphics[width=11.5cm]{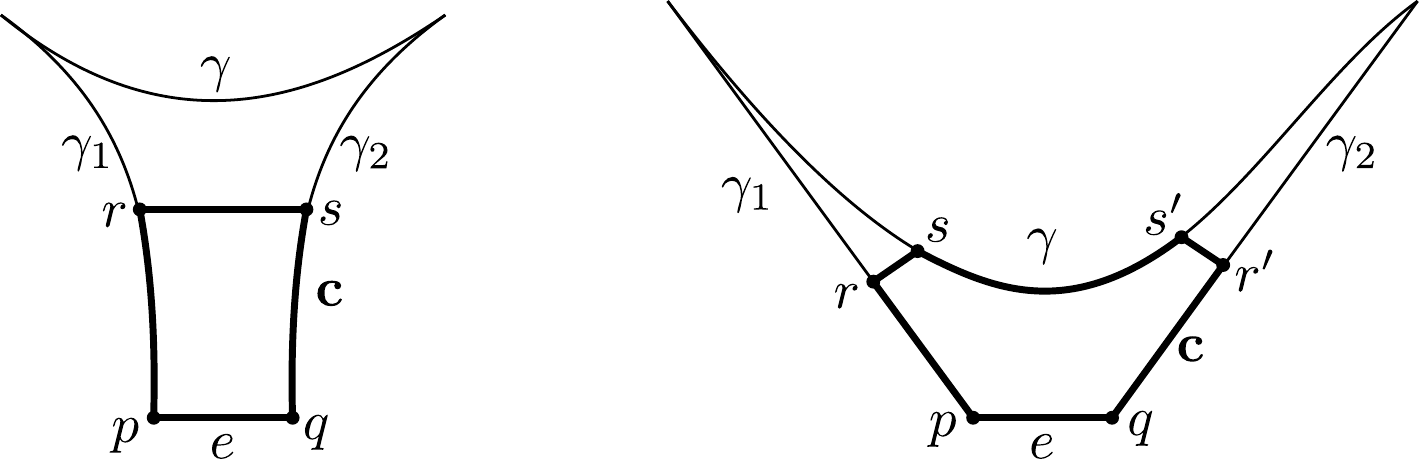}
\caption{}
\label{cutcycle}
 \end{figure}
 
 Consider a point $r$ on $\gamma_1$ at distance $10C+10$ from the final point $p$ of $\gamma_1$, or let $r$ be the starting point of $\gamma_1$ if such point does not exist. We know that $r$ is $C$-close to a point $s$ on either $\gamma_2$ or $\gamma$ (and we set $r=s$ if $r$ is the starting point of $\gamma_1$). In the first case we let $c$ be a quadrangle with vertices $p,r,s,q$ as in the left part of Figure \ref{cutcycle}. The geodesic from $r$ to $s$ cannot contain any edge with both endpoints in $B$ because (either it is trivial or) its length is at most $C$ and one of its endpoints is at distance $10C+10$ from a point in $B$.
 If instead $s$ is on $\gamma$, we pick $r'$ along $\gamma_2$ similarly to $r$. If $r'$ is $C$-close to a point $s'$ in $\gamma_1$ we form a quadrangle as above. 
 Otherwise $r'$ is $C$-close to a point $s'\in\gamma$ and we let $c$ be a hexagon with vertices $p,r,s,s',r',q$ as on the right part of Figure \ref{cutcycle}.
 Again, it is not difficult to show that $e$ is the only edge of $c$ that connects points in $B$.

 Observe now that the cycle $c$ has length bounded by $50C+10$, and its only component labelled by an element of $H_\l\setminus\{1\}$ is the edge $e$. Therefore, by definition of the relative metric $d_B$, 
 we get that $d_B(p,q)<50C+10$. 
This holds for all $p\in\pi_B(x)$ and $q\in\pi_B(y)$, and $C$ only depends on $G$ and $X$, so we are done.
\end{proof}

Let now $B$, $B'$ be distinct cosets, and take points $x,y\in B'$ with $x\neq y$. The geodesic $[x,c(B')]\cup [c(B'),y]$
does not contain $c(B)$, so the previous lemma implies that $\diam_B(\pi_{B}(x)\cup \pi_{B}(y))< D$. 
As a consequence, we easily get the following:

\begin{lemma}\label{smallproj}
If the cosets $B, B'$ are distinct, then
$\diam_B \pi_{B}(B')< D$. In particular, $\diam_B(\pi_{B}(x))< D$ for every $x\in \widehat{G}\setminus \{c(B)\}$, and
$\diam_B(B\cap B')<D$,  $\diam_{B'}(B\cap B')< D$ for any pair of distinct cosets $B, B'$.
\end{lemma}

We also have the following:

\begin{lemma}\label{Rfinite:pre}
 Take $v_0,v_1\in G$. Then the set
 $$
 \{B\in\calB\, |\, \diam_B(\pi_B(v_0)\cup \pi_B(v_1))\geq D\} 
 $$
 is finite.
\end{lemma}
\begin{proof}
By Lemma~\ref{farproj}, if $d_B(\pi_B(v_0),\pi_B(v_1))\geq D$ then any geodesic in $\widehat{G}$ 
joining $v_0$ with $v_1$ contains $c(B)$, as well as a subgeodesic of length $1/2$ centered at $c(B)$. 
Such subgeodesics can intersect at most at their endpoints, 
so the set of cosets described in the statement can contain at most $2\widehat{d}(v_0,v_1)$ elements.
\end{proof}

\subsection*{The BBF axioms}

The projections on hyperbolically embedded subgroups satisfy certain axioms introduced by Bestvina, Bromberg and Fujiwara in
\cite{BBF}, which we will refer to as the \emph{BBF axioms}.
In order to simplify the statement of the theorem below, we restrict
ourselves to the specific case we are interested in, namely cosets of
subgroups of a given group. As opposed to the rest of the Section we restrict here to the case in which the group $G$ is finitely generated and the family $\{H_\l\}_{\l\in\L}$ is finite.
We already pointed out that in this case  each $H_\l$ is finitely generated.

Let, as above, $\calB$ be the collection of the (labelled) cosets of the $H_\lambda$'s in $G$.
We fix a finite
system of generators $\calS_\lambda$ of $H_\l$ for each $\lambda\in\L$, and if $B\in\mathcal{B}$ is labelled by $\l$ we denote by
$\calC(B)$  a copy of the Cayley graph 
${\rm Cay}(H_\l,\calS_\l)$.
For each $B$ let $\widetilde\pi_B:\calB\backslash\{B\}\to\calP(\calC(B))$ be a
function (where $\calP(\calC(B))$ is the collection of all subsets of
$\calC(B)$). Define
$$d_Y(X,Z)=\diam_{\calC(Y)}(\widetilde\pi_Y(X)\cup\widetilde\pi_Y(Z)).$$
Here the diameter is considered with respect to the word metric.
We will say that the family of projections $\{\widetilde\pi_Y\}_{Y\in\calB}$ satisfies the BBF axioms if the following
holds. There exists $\xi < \infty$ so that, using the enumeration in
\cite[Sections 2.1, 3.1]{BBF}:
\begin{itemize}
     \item[(0)] $\diam_{\calC(Y)} (\widetilde\pi_Y(X))<\xi$ for all distinct $X, Y \in \calB$,
     \item[(3)] for all distinct $X, Y, Z \in \calB$ we have
$\min\{d_Y(X,Z),d_Z(X,Y)\}\leq \xi$,
     \item[(4)] $\{Y:d_Y(X,Z)\geq \xi\}$ is a finite set for each
$X,Z\in\calB$.
\end{itemize}

Combining results in the literature, one can obtain the following
theorem, which roughly speaking says that the family of  subgroups $\{H_\lambda\}_{\l\in\L}$ is
hyperbolically embedded in $G$ if and only if one can define projections on the
cosets of $H_\lambda$ satisfying the BBF axioms.

\begin{thm}\label{BBF:thm}
Let $\{H_\lambda\}_{\lambda\in\Lambda}$ be a finite family of finitely generated subgroups of the
finitely generated group $G$, and let $\mathcal{B}$ be the set of the (labelled) cosets of the $H_\l$'s.
\begin{enumerate}
  \item Suppose that it is possible to assign, for each pair of cosets
$Y_1,Y_2\in\mathcal{B}$, a subset $\widetilde\pi_{Y_1}(Y_2)\subseteq
Y_1$ in an equivariant way (i.e. in such a way that
$\widetilde\pi_{gY_1}(gY_2)=g\widetilde\pi_{Y_1}(Y_2)$) and so that the BBF axioms are
satisfied. Then $\{H_\l\}_{\l\in\L}$ is hyperbolically embedded in $G$.
  \item Suppose $\{H_\l\}_{\l\in\L}$ is hyperbolically embedded in $(G,X)$. Then the family of projections
  $\{\pi_Y\}_{Y\in\calB}$ as in Definition \ref{projdefn} satisfies the BBF axioms.
\end{enumerate}
  \end{thm}

\begin{proof}
(1) The set of projections satisfying the BBF axioms can be used to construct a certain metric space out
of $\{\calC(B)\}_{B\in\calB}$. We briefly overview the construction for the sake of
completeness. The details can be found in \cite[Section 3.1]{BBF}.

First, the authors of \cite{BBF} define, using the functions $d_Y$, a certain graph $\calP_K(\calB)$ with vertex set $\calB$. We will not need the precise definition.
Then, they construct the path metric space $\calC(\calB)$
consisting of the union of all $\calC(B)$'s and edges of length 1
connecting all points in $\widetilde\pi_X(Z)$ to
all points in $\widetilde\pi_Z(X)$ whenever $X,Z$ are connected by an edge in
$\calP_K(\calB)$.

As it turns out, $\calC(\calB)$ is hyperbolic relative to $\{\calC(B)\}_{B\in\calB}$ \cite[Theorem
6.2]{Si-metrrh} (even more, it is quasi-tree-graded \cite{Hu-bbf-qtreegr}).
Moreover, the construction of $\calC(\calB)$ is natural in the sense that $G$ acts on $\calC(\calB)$ by isometries. 
The action is such that for each $g\in G$ we have $g(\calC(Y))=\calC(gY)$, and $H_\lambda$ acts on $\calC(H_\lambda)$ by left translations.

In particular, $G$ acts coboundedly on $\calC(\calB)$ in such a way that $\calC(\calB)$ is hyperbolic relative to the orbits of the cosets of the $H_\lambda$'s
which coincide, for an appropriate choice of basepoints, with the copies of the $\calC(B)$'s contained in $\calC(\calB)$.
Also, each $H_\lambda$ acts properly.
Using the characterization of being hyperbolically embedded given in
\cite[Theorem 6.4]{Si-metrrh} (see also \cite[Theorem 4.42]{DGO}) in
terms of actions on a relatively hyperbolic space, we can now
conclude that $\{H_\lambda\}_{\lambda\in\Lambda}$ is hyperbolically embedded in $G$.

(2) Recall from Remark~\ref{finitelygen:rem} that, for every $B\in\calB$, the relative metric $d_B$ and 
the word metric $d_{\calC(B)}$ are bi-Lipschitz equivalent. Therefore,
Axioms (0) and (4) follow respectively  from Lemma \ref{smallproj} and~\ref{Rfinite:pre}.
Let us now show Axiom (3) (cfr. \cite[Lemma 2.5]{Si-contr}). Let $X,Y,Z$ be distinct and suppose that $d_Y(X,Z)>\xi$ (for $\xi$ large enough). We have to show that $d_Z(X,Y)\leq\xi$. Pick $x\in X$ and observe that Lemma \ref{farproj} implies that any geodesic in $\widehat{G}$ from $x$ to $Z$ contains $c(Y)$. In particular, $\pi_Z(x)$ is contained in $\pi_Z(c(Y))$, and the conclusion easily follows (keeping into account Axiom (0)).
\end{proof}

\begin{rem}
Fix the notation of part (2) of the theorem. Since $G$ is finitely generated, by~\cite[Corollary 4.27]{DGO}
we may assume that $\{H_\l\}_{\l\in\L}$ is hyperbolically embedded in $(G,X)$, where $X$ is a (possibly infinite)
set of generators of $G$.
By \cite[Theorem 6.4]{Si-metrrh}, $\Gamma={\rm Cay}(G,X)$ is (metrically)
hyperbolic relative to the cosets of the $H_\lambda$'s. It is observed in
\cite[Lemma 4.3]{MS-prodtrees} that the BBF axioms are satisfied in this
setting when the $\pi_Y$'s are defined as the closest point projections with respect to the metric of $\Gamma$.
Hence, part (2) of the theorem also holds for this other set of projections.
On the other hand, it could be shown using techniques from \cite{Si-proj} that projections as in Definition \ref{projdefn} and closest point projections in $\Gamma$ coarsely coincide
(but we will not need this).
\end{rem}

\section{The trace of a simplex on a coset}\label{trace:sec}
Throughout this section, we fix a group $G$ with a hyperbolically embedded family of subgroups 
$\{H_\lambda\}_{\lambda\in\Lambda}$. We also denote by $D$ the constant provided by Lemma~\ref{farproj}.

For every $B\in\calB$,
if $\overline{g}=(g_0,\ldots,g_n)\in \widehat{G}^{n+1}$, then we set 
$$\diam_B(\pi_B(\overline{g}))=\diam_B (\pi_B(g_0)\cup\ldots\cup \pi_B(g_n))\ .$$
In particular, if $g_i\in B$ for every $i$, then $\diam_B(\overline{g})=\diam_B (\{g_0,\ldots,g_n\})$.

We begin the section with a definition that is a version, suited to our context, of \cite[Definition 3.1,3.6]{HullOsin}.
\begin{defn}
Let $B\in\calB$
and let $v_0,v_1$ be vertices of $\widehat{G}$. 
We say that $v_0,v_1\in \widehat{G}$ are \emph{separated} by $B$ if  $\diam_B(\pi_{B}(v_0)\cup \pi_{B}(v_1))\geq D$,
and we denote by $\calS(v_0,v_1)$ the set of cosets that separate $v_0$ from $v_1$.
Let $\overline{g}=(g_0,\ldots,g_n)\in \widehat{G}^{n+1}$. 
A coset $B$ is \emph{relevant} for $\overline{g}$ if $\diam_B(\pi_B(\overline{g}))\geq 2D$, 
and we denote by $\calR(\overline{g})$ 
the set of all relevant cosets for $\overline{g}$. 
\end{defn}

As the name suggests, if a coset $B$ separates $v_0$ from $v_1$, then by Lemma~\ref{farproj}
every geodesic joining $v_0$ to $v_1$ must contain $c(B)$ and intersect $B$ unless $v_0=v_1=c(B)$. 
Moreover, we obviously have $\calR(v_0,v_1)\subseteq \calS(v_0,v_1)$.

For every pair of vertices $v_0,v_1$  of $\widehat{G}$,
we are going to endow $\calS(v_0,v_1)$ with a total ordering $<$
(so  $\calR(v_0,v_1)$ will be endowed with a total ordering as well).

Fix vertices $v_0,v_1$ of $\widehat{G}$, let $B_0,B_1$ be cosets in $\calS(v_0,v_1)$, and take any geodesic $\gamma$ starting
at $v_0$ and ending at $v_1$. By Lemma~\ref{farproj} we know that 
$\gamma$ must pass through $c(B_i)$, $i=1,2$, so 
$\widehat{d}(v_0,B_i)=\widehat{d}(v_0,\In_\gamma(B_i))$. In particular, we have that 
either $\widehat{d}(v_0,B_0)<\widehat{d}(v_0,B_1)$ (and along every geodesic
starting at $v_0$ and ending at $v_1$ the point $c(B_1)$ follows $c(B_0)$), 
or $\widehat{d}(v_0,B_1)<\widehat{d}(v_0,B_0)$ . We stipulate that $B_0<B_1$ in $\calS(v_0,v_1)$ in the first case,
while $B_1<B_0$ in the second case. 
It follows from the very definitions that $\calS(v_0,v_1)=\calS(v_1,v_0)$
as (unordered) sets. However,
$B_0<B_1$ in $\calS(v_0,v_1)$ if and only if $B_1<B_0$ in $\calS(v_1,v_0)$.

\begin{lemma}\label{Rfinite}
 Take $v_0,v_1\in \widehat{G}$ and 
$\overline{g}\in \widehat{G}^{n+1}$. Then the sets $\calS(v_0,v_1)$ and $\calR(\overline{g})$ are finite.
\end{lemma}
\begin{proof}
The first statement is just a restatement of Lemma~\ref{Rfinite:pre}, while the second one 
follows from the fact that, if $\overline{g}=(g_0,\ldots,g_n)$, then
$\calR(\overline{g})
=\bigcup_{i\neq j}\calR(g_i,g_j)\subseteq \bigcup_{i\neq j}\calS(g_i,g_j)$. 
\end{proof}

\begin{lemma}\label{linearord:lemma}
Take points  $v_0,v_1\in \widehat{G}$ and cosets 
$B_0,B_1$ in $\calS(v_0,v_1)$ such that $B_0<B_1$.
Then $\pi_{B_1} (v_0)=\pi_{B_1}(c(B_0))$ and
$\pi_{B_0}(c(B_1))=\pi_{B_0}(v_1)$. 
\end{lemma}
\begin{proof}
\begin{figure}[h]
 \includegraphics[width=8cm]{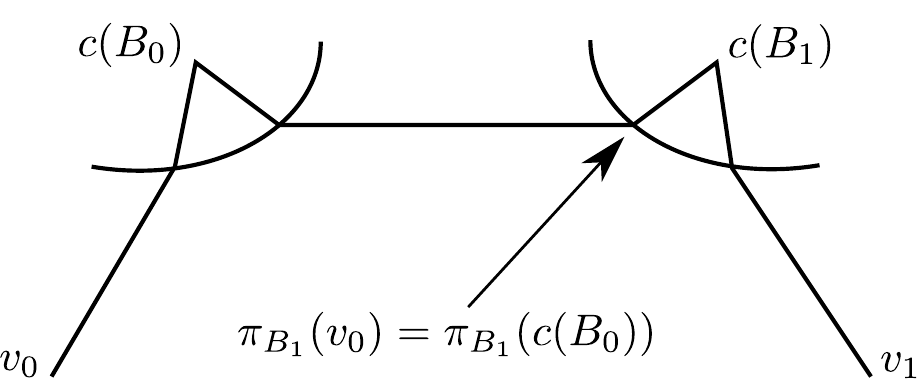}
\caption{}
\label{orderandproj}
\end{figure}
By symmetry, it is sufficient to show that $\pi_{B_1} (v_0)=\pi_{B_1}(c(B_0))$.
Take a geodesic $\gamma$ joining $v_0$ to $v_1$
(see Figure~\ref{orderandproj}). Then 
$c(B_1)$ follows $c(B_0)$ along $\gamma$, so
$\widehat{d}(v_0,c(B_1))=\widehat{d}(v_0,c(B_0))+\widehat{d}(c(B_0),c(B_1))$,
and the concatenation of a geodesic between $v_0$ and $c(B_0)$ with any geodesic 
between $c(B_0)$ and $c(B_1)$ is itself a geodesic. This implies that
$\pi_{B_1}(c(B_0))\subseteq \pi_{B_1} (v_0)$. In order to conclude it is sufficient to show
that
every geodesic joining $v_0$ with $c(B_1)$ must contain $c(B_0)$.
Suppose by contradiction that the geodesic $\gamma$ joins $v_0$ to $c(B_1)$ and
avoids $c(B_0)$. Since $B_0<B_1$, there exists a geodesic $\gamma'$ joining $c(B_1)$
to $v_1$ and avoiding $c(B_0)$. 
Since every geodesic joining $v_0$ to $v_1$ passes through $c(B_1)$, 
we have $\widehat{d}(v_0,v_1)=\widehat{d}(v_0,c(B_1))+\widehat{d}(c(B_1),v_1)$, so
the concatenation
$\gamma * \gamma'$ is itself a geodesic. But $\gamma *\gamma'$ joins $v_0$ to $v_1$ without passing
through $c(B_0)$, a contradiction. 
\end{proof}

\begin{prop}\label{3points:prop}
 Let $g_0,g_1,g_2$ be elements of $\widehat{G}$. Then there exist at most two cosets $B\in \calR(g_0,g_1)$ 
 such that $\pi_B(g_2)\neq \pi_B(g_0)$ and $\pi_B(g_2)\neq \pi_B(g_1)$.
 \end{prop}

\begin{proof}
 Let us enumerate the elements $B_1,\ldots,B_k$ of $\calR(g_0,g_1)$
 in such a way that $B_i <B_{i+1}$.
 We set
  $$
 \Omega=\{i\in \{1,\ldots,k\}\, |\, \diam_{B_i}(\pi_{B_i}(g_1), \pi_{B_i}(g_2))\leq D\}\ .
 $$
 Moreover, we set $i_0=\max \Omega$  if $\Omega\neq\emptyset$, and
 $i_0=0$ otherwise. In the following arguments we will  use the obvious fact that,
 if $A_1,A_2,A_3$ are non-empty subsets of a metric space, then
 $$
 \diam (A_1\cup A_2)\leq \diam (A_1\cup A_3)+\diam (A_2\cup A_3)\ .
 $$

In order to conclude, it is sufficient to prove Claims (1) and (2) below.
  
 \smallskip 
 
\noindent{\bf Claim 1.}  $\pi_{B_i}(g_2)=\pi_{B_i}(g_1)$ for every $i< i_0$.

We may suppose that $i_0>1$, otherwise the statement is empty.
Since $i<i_0$, 
Lemma~\ref{linearord:lemma} implies that
$\pi_{B_{i_0}}(c(B_i))=\pi_{B_{i_0}}(g_0)$,
so
\begin{equation}\label{triang1}
\diam_{B_{i_0}}(\pi_{B_{i_0}}(c(B_i)), \pi_{B_{i_0}}(g_1))=
\diam_{B_{i_0}}(\pi_{B_{i_0}}(g_0), \pi_{B_{i_0}}(g_1))\geq 2D\ ,
\end{equation}
where the last inequality is due to the fact that
$B_{i_0}\in\calR(g_0,g_1)$. But $i_0\in\Omega$, so
$\diam_{B_{i_0}}(\pi_{B_{i_0}}(g_1), \pi_{B_{i_0}}(g_2))\leq D$. Together with~\eqref{triang1},
this implies that $\diam_{B_{i_0}}(\pi_{B_{i_0}}(c(B_i)), \pi_{B_{i_0}}(g_2))\geq D$, i.e.~$B_{i_0}\in \calS(c(B_i),g_2)$.
Of course also $B_i$ belongs to $\calS(c(B_i),g_2)$, and $B_i<B_{i_0}$ in $\calS(c(B_i),g_2)$, so Lemma~\ref{linearord:lemma}
implies that $\pi_{B_i}(g_2)=\pi_{B_i}(c(B_{i_0}))=\pi_{B_i}(g_1)$,
where the last equality is due to the fact that 
$B_i<B_{i_0}$ in $\calR(g_0,g_1)$.

\smallskip

\noindent{\bf Claim 2.}  $\pi_{B_i}(g_2)=\pi_{B_i}(g_0)$ for every $i> i_0+1$.
 
 We set $i_1=i_0+1$ for convenience. We may suppose that $i_1<k$, otherwise the statement is empty.
 Since $i>i_1$, Lemma~\ref{linearord:lemma} implies that
$\pi_{B_{i_1}}(c(B_i))=\pi_{B_{i_1}}(g_1)$, so
 $$
 \diam_{B_{i_1}}(\pi_{B_{i_1}}(g_2), \pi_{B_{i_1}}(c(B_i)))=
 \diam_{B_{i_1}}(\pi_{B_{i_1}}(g_2), \pi_{B_{i_1}}(g_1)) > D\ ,
 $$
where the last inequality is due to the fact that
$i_1\notin \Omega$. So 
$B_{i_1}\in \calS(g_2, c(B_i))$.
Of course also $B_i$ belongs to $\calS(g_2, c(B_i))$, and $B_{i_1}<B_i$ in $\calS(g_2,c(B_i))$, so Lemma~\ref{linearord:lemma}
implies that $\pi_{B_i}(g_2)=\pi_{B_i}(c (B_{i_1}))=\pi_{B_i}(g_0)$, where the last equality is due to the fact that
$B_i>B_{i_1}$ in $\calR(g_0,g_1)$.
\end{proof}

\subsection*{The trace of a simplex}
Let us now come back to our original extension problem.
In order to extend a cochain defined on $H_\lambda$ to a cochain defined on the whole of $G$ we need to be able to project a
simplex with vertices in $G$ onto a simplex (or, at least, onto a chain) supported in $H_\lambda$ (or, more in general, on a coset of $H_\lambda$).
To this aim we give the following definition. 

\begin{defn}
Let $\overline{g}=(g_0,\ldots,g_n)\in G^{n+1}$ be any simplex, and fix a coset $B$. Then we define the \emph{trace} 
$\tr_n^{B}(\overline{g})\in C_n(B)$
of $\overline{g}$
on $B$ by setting $\tr_n^{B}(\overline{g})=0$ if $B\notin \calR(\overline{g})$ and
$$
\tr_n^{B}(\overline{g})=\frac{1}{\prod_{j=0}^n |\pi_{B}(g_j)|} \sum_{h_j\in \pi_{B}(g_j)} (h_0,\ldots,h_n)
$$
if $B\in\mathcal{R}(\overline{g})$.
In other words, the trace of $\overline{g}$ on $B$ is either null, if $B$
is not relevant for $\overline{g}$, or the average of the simplices
obtained by projecting $\overline{g}$ onto $B$. The map $\tr
^{B}_n$ uniquely extends to a linear map $\tr
^{B}_n\colon C_n(G)\to C_n(B)$. By construction, this map is norm non-increasing.
\end{defn}

The strategy to extend to the whole of $G$ a cochain $\varphi_\lambda$ defined on $H_\lambda$
is clear: for every $\overline{g}\in G^{n+1}$, we just add up the sum of the values
of $\varphi_\lambda$ on the traces of $\overline{g}$ on the cosets of $H_\lambda$.
In order to check that this procedure indeed takes quasi-cocycles to quasi-cocycles we need to prove
that trace operators do not behave too wildly with respect to taking coboundaries.
This boils down to showing that the trace operator defined on chains is ``almost'' a chain map, in a sense
that is specified in Proposition~\ref{almost-chain:prop}. We warn the reader that there is no hope to replace traces with  genuine chain maps: in fact, if this were possible, 
then, at least in the case when $\{H_\l\}_{\l\in\L}=\{H\}$ consists of a single subgroup, it would be easy to prove
that the restriction map $H^n_b(G,\R)\to H^n_b(H,\R)$ is surjective. However, as anticipated in the introduction,
this is not true in general (see Proposition~\ref{prop:7.2}).

\smallskip 

Recall from Section~\ref{background:sec} that a chain $c\in C_n(G)$ is \emph{degenerate} if
$\alt_n(c)=0$.

\begin{lemma}\label{natmost:lem}
Let  $n\geq 2$ and take $\overline{g}\in G^{n+1}$. Then there exist at most $n(n+1)$ cosets 
$B\in\calB$ 
such that
$\tr^B_n (\overline{g})$ is \emph{not} degenerate. 
\end{lemma}
\begin{proof}
Set $\overline{g}=(g_0,\ldots,g_n)$, and suppose that 
$B\in\calB$ 
is such that
$\tr^B_n (\overline{g})$ is not degenerate. Of course $B\in\calR(\overline{g})$, so
there exist $i,j\in \{0,\ldots,n\}$ such that $B\in\calR(g_i,g_j)$
(in particular, $i\neq j$).
Observe now that, if there exists $k\in \{0,\ldots,n\}$, $k\notin \{i,j\}$ such that
$\pi_B(g_k)=\pi_B(g_i)$ or $\pi_B(g_k)=\pi_B(g_j)$, then $\tr_n^B(\overline{g})$ is degenerate. 
So the conclusion follows from Proposition~\ref{3points:prop}: the number of cosets such that $\tr^B_n (\overline{g})$ is \emph{not} degenerate is  at most twice the number of pairs $(i,j)$ of distinct elements of $\{0,\ldots,n\}$.
\end{proof}

\begin{defn}
Let $\overline{g}=(g_0,\ldots,g_n)\in G^{n+1}$.
We say that $\overline{g}$ is \emph{small} if there exists a coset $B\in\calB$
such that $\overline{g}$ is supported in $B$ and $\diam_B (\overline{g})< 2D$. A chain $c\in C_n(G)$ 
is small if
it is a linear combination of small simplices.
\end{defn}

\begin{lemma}\label{small:lemma}
\begin{enumerate}
\item Take $\overline{g}\in B^{n+1}$ for some $B\in\mathcal{B}$. Then $\overline{g}$ is small if and only if $\diam_B(\overline{g})< 2D$.
\item For every $\lambda\in\Lambda$, the set of small $n$-simplices 
supported in $H_\lambda$ is $H_\lambda$-invariant.
\item The number of $H_\lambda$-orbits of small $n$-simplices supported in $H_\lambda$ is finite.
\item Take $\overline{g}\in B^{n+1}$ for some $B\in\mathcal{B}$. Then $\calR(\overline{g})=\emptyset$ if $\overline{g}$ is small,
and $\calR(\overline{g})=\{B\}$ otherwise.
\end{enumerate}
\end{lemma}
\begin{proof}
(1): If $\overline{g}\subseteq B^{n+1}$ is small, then there exists $B'\in\calB$ such that
$\overline{g}\subseteq (B')^{n+1}$ and $\diam_{B'}(\overline{g})<2D$. If $B=B'$ we are done. Otherwise
$B\neq B'$ and $\overline{g}$ is supported in the intersection $B\cap B'$, so $\diam_B(\overline{g})<D< 2D$
by
Lemma~\ref{smallproj}. The converse implication is obvious.

Since the metric $d_{H_\lambda}$ on $H_\lambda$
is locally finite, points (2) and (3) immediately follow from (1). Finally, let $\overline{g}\in B^{n+1}$.
By Lemma~\ref{smallproj}, if $B'\in\calB\setminus\{B\}$, then $\diam_{B'}(\pi_{B'}(B))< D$,
so $B'\notin \calR(\overline{g})$. Moreover, point (1) implies that $B\in \calR(\overline{g})$ if and only if
$\overline{g}$ is not small.
This concludes
the proof of the lemma.
\end{proof}

The following result shows that the trace operators commute with the boundary operator, up to small chains.

\begin{prop}\label{almost-chain:prop}
Fix $\overline{g}\in G^{n+1}$, $n\geq 2$. Then the chain
$$\partial \tr^B_n (\overline{g})-\tr^B_{n-1}(\partial \overline{g})$$ 
is small for every $B\in\calB$.
\end{prop}
\begin{proof}
If $B$ is not relevant for $\overline{g}$, then it is not relevant for 
any face of $\overline{g}$, so $\partial \tr^B_n (\overline{g})=\tr^B_{n-1}(\partial \overline{g})=0$.
So, let $B\in\calR(\overline{g})$.
The equality  $\partial( \tr^B_n (\overline{g}))= \tr^B_{n-1}(\partial \overline{g}) $ may fail only 
when there exist some cosets in $\calR(\overline{g})$ which are not relevant for some face
of $\overline{g}$. 
More precisely, an easy computation shows that
$$
\partial( \tr^B_n (\overline{g}))- \tr^B_{n-1}(\partial \overline{g})=
\sum_{i\in\Omega} (-1)^i c_i\ ,
$$
where $\Omega=\{i\in\{0,\ldots,n\}\, |\, B\notin \calR(\partial_i \overline{g})\}$ and
$$
c_i= \frac{1}{\prod_{j\neq i} |\pi_B(g_j)|}\sum_{\begin{array}{c}h_l\in \pi_B(g_l)\\ l\neq i\end{array}}(h_0,\ldots,h_n)\ ,
$$
so
we are left to show that $c_i$ is small for every $i\in\Omega$.  
However, if $i\in\Omega$, then $B\notin \calR(g_j,g_k)$
for every $j,k\in\{0,\ldots,n\}$ such that $i,j,k$ are pairwise disjoint. In other words,
for any such $j,k$ we have 
$\diam_B(\pi_B(g_j),\pi_B(g_k))< 2D$. This 
implies in turn that $c_i$ is small, hence the conclusion.
\end{proof}

\section{Proof of Theorem~\ref{main:thm}}\label{proof:sec}

Let $\{H_\lambda\}_{\lambda\in\Lambda}$ be a hyperbolically embedded family of subgroups of the group $G$. Moreover,
let $V$ be a normed $\mathbb{R}[G]$-module, and for every $\lambda\in\Lambda$ let $U_\lambda$ be an $\mathbb{R}[H_\lambda]$-submodule
of $V$. This section is devoted to the proof of Theorem~\ref{mainc}, which specializes to Theorem~\ref{main:thm}
in the case when $U_\lambda=V$ for every $\lambda\in\Lambda$.
In fact, we will deduce  Theorem~\ref{mainc} from Theorem~\ref{maincc} below, which deals with extensions of alternating cochains
that need not be quasi-cocycles.

We first fix some notation. For every $\varphi=(\varphi_\lambda)_{\lambda\in\Lambda}\in \oplus_{\lambda\in\Lambda} C_{\alt}^n(H_\lambda, U_\l)^{H_\lambda}$ we denote by $\delta^n\varphi$ the element 
$\delta^n\varphi=(\delta^n\varphi_\lambda)_{\lambda\in\Lambda}\in \oplus_{\lambda\in\Lambda} C_{\alt}^{n+1}(H_\lambda, U_\l)^{H_\lambda}$.
We set
$$
K(\varphi)=\max \{\|\varphi_\lambda(\ov s)\|_{U_\lambda},\ \lambda\in\Lambda,\ \ov s\subseteq H_\lambda^{n+1}\ {\rm small}\}\ .
$$
Since  $\varphi_\lambda=0$ for all but a finite number of indices, and the number of $H_\lambda$-orbits of small
simplices in $H_\lambda^{n+1}$ is finite (see Lemma~\ref{small:lemma}), the value $K(\varphi)$ is well-defined and finite.
We also set
$$
\|\varphi\|_\infty=\max_{\lambda\in\Lambda} \|\varphi_\lambda\|_\infty\in [0,\infty]\, ,\quad
\|\delta^n\varphi\|_\infty=\max_{\lambda\in\Lambda} \|\delta^n\varphi_\lambda\|_\infty\in [0,\infty]\ .
$$
In particular, 
$\|\delta^n\varphi\|_\infty<\infty $ if and only if every $\varphi_\lambda$ is a quasi-cocycle.
If this is the case, then we define the defect $D(\varphi)$ of $\varphi$ by setting
$$
D(\varphi)=\|\delta^n\varphi\|_\infty= \max_{\lambda\in\Lambda} D(\varphi_\lambda)\, .
$$

\begin{thm}\label{maincc}
For every $n\geq 1$, there exists a linear map
 $$
 \Theta^n\colon \bigoplus_{\lambda\in\Lambda} C_{\alt}^n(H_\lambda,U_\lambda)^{H_\lambda}\to C_{\alt}^n(G,V)^G 
$$
such that, for every $\varphi=(\varphi_\lambda)_{\lambda\in\Lambda}\in \bigoplus_{\lambda\in\Lambda} C_{\alt}^n(H_\lambda,U_\lambda)^{H_\lambda}$ and for every 
$\lambda\in\Lambda$, the following conditions hold:
\begin{enumerate}
\item\label{1cond}
$\Theta^n(\varphi)(H_\lambda^{n+1})\subseteq U_\lambda\ ;$\vspace{-.3cm}\\
\item\label{2cond}
$
\sup_{\overline{h}\in H^{n+1}_\lambda} \|\Theta^n(\varphi)(\overline{h})-\varphi_\lambda(\overline{h})\|_{U_\lambda}\leq K(\varphi)\ ;
$\vspace{-.3cm}\\
\item\label{3cond}
If $n\geq 2$ then $
\|\Theta^n(\varphi)\|_\infty \leq n(n+1)\cdot \|\varphi\|_\infty\ ;
$\vspace{-.3cm}\\
\item\label{4cond}
$
\|\delta^n\Theta^n(\varphi)-\Theta^{n+1}(\delta^n\varphi)\|_\infty\leq 2(n+1)(n+2)K(\varphi)\ .
$
\end{enumerate}
 \end{thm}
 \begin{proof}
For every coset $B\in\calB$ we define a cochain $\varphi_B\in C^n_{\alt}(B,V)$ as follows: if $\overline{b}=(b_0,\ldots,b_n)\in B^{n+1}$, then 
$$
\varphi_B(\overline{b})=b_0\cdot \varphi_\lambda (1,b_0^{-1}b_1,\ldots,b_0^{-1}b_n)\ ,
$$
where $\lambda\in\Lambda$ is the label of $B$ and $g\in G$ is such that   $B=gH_\lambda$ 
(the fact that $\varphi_B$ is indeed  alternating is easily checked). 
Then, we define a new cochain $\varphi_B'\in C^n_{\alt}(B,V)$
by setting
$$
\varphi'_B(\overline{b})=\left\{\begin{array}{cl}
0 & {\rm if}\ \overline{b}\ {\rm is\ small}\\
\varphi_B(\overline{b}) & {\rm otherwise.}\end{array}\right.
$$
The new cochain $\varphi'_B$ stays at bounded distance from $\varphi_B$. More precisely, it follows from 
the definitions that
$$
\|\varphi'_B-\varphi_B\|_\infty\leq K(\varphi)
$$
for every $B\in\calB$. If $B=H_\lambda$ for some $\lambda\in\Lambda$, then we set $\varphi'_\lambda=\varphi'_{H_\lambda}$, thus getting
that $\|\varphi'_\lambda-\varphi_\lambda\|_\infty\leq K(\varphi)$.




We are now ready to define the element
$
\Phi=\Theta ((\varphi_\lambda)_{\lambda\in\Lambda})\in C_{\alt}^n(G,V)^G
$
as follows:
\begin{equation}\label{Phi}
\Phi(\overline{g})=\sum_{B\in \calB} \varphi'_B(\tr_n^B(\overline{g}))\ .
\end{equation}
Recall that $\tr_n^B(\overline{g})=0$ whenever $B\notin \calR(\overline{g})$, so by Lemma~\ref{Rfinite}
the sum on the right-hand side of~\eqref{Phi} is finite. 
It is easy to check that 
$\Phi$ is alternating. Moreover, the $H_\lambda$-invariance of each $\varphi_\lambda$ and
the $G$-invariance of the set of small simplices readily imply that $\Phi$ is indeed $G$-invariant. 

In order to show that conditions~\eqref{1cond} and~\eqref{2cond} are satisfied it is sufficient to show that
the restriction of $\Phi$ to $H_\lambda$  coincides with $\varphi'_\lambda$.
So, suppose that $\overline{g}\in G^{n+1}$ is supported in $H_\lambda$. If $\overline{g}$ is small,
then by Lemma~\ref{small:lemma}--(4) we have
$\varphi'_\lambda(\overline{g})=\Phi(\overline{g})=0$. On the other hand, 
if $\overline{g}$ is not small, then $\calR(\overline{g})=\{H_\lambda\}$ again
by Lemma~\ref{small:lemma}--(4). Moreover,  we obviously have
$\tr_n^{H_\lambda}(\overline{g})=\overline{g}$, so again
$\Phi(\overline{g})=\varphi'_\lambda(\overline{g})$.

Let us now suppose that each $\varphi_\lambda$ is bounded,
and 
observe that for every $B\in\calB$ we have $\|\varphi_B'\|_\infty\leq \|\varphi_B\|_\infty\leq \|\varphi\|_\infty$.
We fix an element $\overline{g}\in G^{n+1}$.
Since $\|\tr_n^B(\overline{g}))\|_1\leq 1$,
for every $B\in\calB$ we have
$\|\varphi'_B(\tr_n^B(\overline{g}))\|_V\leq \|\varphi\|_\infty$. Moreover,
since $\varphi'_B$ is alternating, by Lemma~\ref{natmost:lem} there are at most $n(n+1)$ cosets $B\in\calB$ such that
$\varphi'_B(\tr_n^B(\overline{g}))\neq 0$, so 
$$
\|\Phi(\overline{g})\|_V=\left\|\sum_{B\in\calB} \varphi'_B(\tr_n^B(\overline{g}))\right\|\leq n(n+1)\|\varphi\|_\infty\ .
$$ 
This proves condition~\eqref{3cond}.

Let us now concentrate our attention on condition~\eqref{4cond}.
In order to compare $\Theta^{n+1}(\delta^n\varphi)$ with $\delta^n\Theta^n(\varphi)$ we first observe that
$(\delta^n\varphi_B)'$ 
does \emph{not} coincide in general with $\delta^n\varphi'_B$. In fact, let us fix an $(n+1)$-simplex $\overline{b}\in B^{n+2}$.
If $\overline{b}$ is small, then also every face of $\overline{g}$ is small, and this readily implies
that $(\delta^n\varphi_B)'(\overline{b})=\delta^n\varphi'_B(\overline{b})=0$. On the other hand, suppose that $\overline{b}$
is not small, and  set
$$
\Omega=\{i\in\{0,\ldots,n+1\}\, |\, \partial_i\overline{b}\ {\rm is\ small}\}\ .
$$
Since $\overline{b} $ is not small, there exist distinct vertices $b_{i_0}, b_{i_1}$ of $\overline{b}$
such that $d_B(b_{i_0},b_{i_1})\geq 2D$. This implies that $\partial_i\overline{b}$ is not small
for every $i\notin\{i_0,i_1\}$, so 
$|\Omega|\leq 2$, and 
$$
\left\|((\delta^n\varphi_B)'-\delta^n\varphi_B')(\overline{b})\right\|_V=\left\|\sum_{i\in\Omega} \varphi_B(\partial_i \overline{b})\right\|_V\leq 2K(\varphi)\ .
$$
We have thus proved that,
for every $B\in\calB$, we have
\begin{equation}\label{normadiff}
\left\|(\delta^n\varphi_B)'-\delta^n\varphi_B'\right\|\leq 2K(\varphi)\ .
\end{equation}

Let us now take any simplex $\overline{g}\in G^{n+2}$. 
Since $\varphi'_B$ vanishes on small chains supported in $B$, 
Proposition~\ref{almost-chain:prop} implies that  
$
\varphi'_B(\tr_n^B (\partial \overline{g}))=\varphi'_B(\partial \tr_{n+1}^B(\overline{g})) 
$
for every $B\in\calB$, so
\begin{align*}
\delta^n\Theta^n(\varphi)(\overline{g}) & =\Theta^n(\varphi)(\partial\overline{g})=
\sum_{B\in\calB} \varphi_B'(\tr_n^B(\partial\overline{g}))=
\sum_{B\in\calB} \varphi'_B(\partial \tr_{n+1}^B(\overline{g})) \\ & =
\sum_{B\in\calB} \delta^n\varphi_B'(\tr_{n+1}^B(\overline{g}))\ .
\end{align*}
On the other hand, we have
$$
\Theta^{n+1}(\delta^n\varphi)(\overline{g})=\sum_{B\in\calB} (\delta^n\varphi_B)'(\tr^B_{n+1}(\overline{g}))\ ,
$$
so
\begin{equation}\label{formulafinale}
\Theta^{n+1}(\delta^n\varphi)(\overline{g})-\delta^n\Theta^n(\varphi)(\overline{g})=
\sum_{B\in\calB} ((\delta^n\varphi_B)'-\delta^n\varphi_B')(\tr^B_{n+1}(\overline{g}))\ .
\end{equation}
Being alternating, the cochain $(\delta^n\varphi_B)'-\delta^n\varphi_B'$ 
vanishes on degenerate chains supported in $B$. On the other hand, recall from
Lemma~\ref{natmost:lem} that
$\tr_{n+1}^B(\overline{g})$ is not degenerate on at most $(n+1)(n+2)$ cosets $B\in\calB$.
Therefore, since $\|\tr_{n+1}^B(\overline{g})\|_1\leq 1$ for every $B\in\calB$, from equation~\eqref{formulafinale}
and inequality~\eqref{normadiff} we get 
$$
\left\|\delta^n\Theta^n(\varphi)(\overline{g})-\Theta^{n+1}(\delta^n\varphi)(\overline{g})\right\|\leq 2(n+1)(n+2)K(\varphi)\ .
$$
This proves condition~\eqref{4cond},
and concludes the proof of the Theorem.
\end{proof}

By considering the restriction to quasi-cocycles of the map $\Theta^n$ constructed in the previous theorem, 
we obtain the following result, which in turn implies Theorem~\ref{main:thm}:

\begin{thm}\label{mainc}
 For every $n\geq 1$, there exists a linear map
 $$
 \Theta^n\colon \bigoplus_{\lambda\in\Lambda} \QCa^n(H_\lambda,U_\lambda)^{H_\lambda}\to \QCa^n(G,V)^G 
$$
such that, for every $\varphi=(\varphi_\lambda)_{\lambda\in\Lambda}\in \bigoplus_{\lambda\in\Lambda} \QCa^n(H_\lambda,U_\lambda)^{H_\lambda}$ and for every 
$\lambda\in\Lambda$, we have $\Theta^n(\varphi)(H_\lambda)\subseteq U_\lambda$ and 
$$
\sup_{\overline{h}\in H^{n+1}_\lambda} \|\Theta^n(\varphi)(\overline{h})-\varphi_\lambda(\overline{h})\|_{U_\lambda}\leq K(\varphi)\ ,
$$
$$
D(\Theta^n(\varphi))\leq (n+1)(n+2)(D(\varphi)+2K(\varphi))\ .
$$
 \end{thm}
\begin{proof}
We are only left to prove 
the estimate on the defect of $\Theta^n(\varphi)$ (which implies that $\Theta^n$ takes indeed quasi-cocycles into quasi-cocycles). However,
by Theorem~\ref{maincc} we have
\begin{align*}
D(\Theta^n(\varphi))& =\|\delta^n(\Theta^n(\varphi))\|_\infty\leq \|\delta^n(\Theta^n(\varphi))-\Theta^{n+1}(\delta^n\varphi)\|_\infty+
\|\Theta^{n+1}(\delta^n\varphi)\|_\infty\\ & \leq 
2(n+1)(n+2)K(\varphi)+ (n+1)(n+2)\|\delta^n\varphi\|_\infty\\ & =
(n+1)(n+2)(2K(\varphi)+D(\varphi))\ .
	\end{align*}
\end{proof}

\section{Applications to bounded cohomology}\label{nonvanishing:sec}
This section is devoted to 
some  applications of Theorem~\ref{main:thm} to bounded cohomology. Throughout the section, we fix  a  hyperbolically embedded family of subgroups $\{H_\lambda\}_{\lambda\in\Lambda}$ 
of a group $G$. We also fix a normed $\mathbb{R}[G]$-space $V$ and, for every $\lambda\in\L$, an $H_\l$-invariant
submodule $U_\l$ of $V$.
The inclusion $C^n_b(H_\lambda,U_\lambda)\to C^n_b(H_\lambda,V)$ induces
a map $i^n_\lambda\colon H^n_b(H_\lambda,U_\lambda)\to H^n_b(H_\lambda,V)$, which restricts to a map $EH^n_b(H_\lambda,U_\lambda)\to EH^n_b(H_\lambda,V)$. 
Finally, for every $\l\in\L$ we denote by  $\res^n_\lambda\colon H^n_b(G,V)\to H^n_b(H_\lambda,V)$ the restriction map. 

The following result provides a 
generalization of Corollary~\ref{boundedco:cor}.

\begin{prop}\label{boundedcoc:prop}
Let $n\geq 2$.
For every element $(\alpha_\lambda)_{\lambda\in\Lambda}\in\bigoplus_{\lambda\in\Lambda}  EH^n_b(H_\lambda,U_\lambda)$, there exists $\alpha\in EH^n_b(G,V)$ such that $\res^n_\lambda(\alpha)=i^n_\lambda(\alpha_\lambda)$ for every $\lambda\in\Lambda$.
\end{prop}
\begin{proof}
Recall that bounded cohomology can be computed from the complex of alternating bounded cochains, so, for every $\lambda\in\Lambda$, we may choose an alternating representative
$a_\lambda\in C^n_{b,\alt}(H_\lambda,U_\lambda)^{H_\lambda}$ of $\alpha_\lambda$. Since $\alpha_\lambda\in EH^n_b(H_\lambda,U_\lambda)^{H_\lambda}$,
we have $a_\lambda=\delta^{n-1} \varphi_\lambda$ for some $\varphi_\lambda\in \QC^{n-1}(H_\lambda,U_\lambda)^{H_\lambda}$. We have
$\delta^{n-1}\alt^{n-1}(\varphi_\lambda)=\alt^n(\delta^{n-1}\varphi_\lambda)=a_\lambda$, so, up to replacing
$\varphi_\lambda$ with $\alt^{n-1}(\varphi_\lambda)$, we may suppose that
$\varphi_\lambda$ is alternating for every $\lambda\in\Lambda$.

We now consider the quasi-cocycle $\Phi=\Theta^{n-1}(\oplus_{\lambda\in\Lambda} \varphi_\lambda)\in \QCa^{n-1}(G,V)^G$, where
$\Theta^{n-1}$ is the map described in Theorem~\ref{mainc}, and we set $\alpha=[\delta^{n-1}\Phi]\in EH^n_b(G,V)$.
In order to conclude it is sufficient to observe that, by Theorem~\ref{mainc}, $\Phi|_{H_\lambda^{n}}$ and $\varphi_\lambda$
differ by a bounded cochain for every $\lambda\in \L$.
\end{proof}

The following result sharpens
Proposition~\ref{boundedcoc:prop} under additional assumptions. 

\begin{prop}\label{inverse2:thm}
Let us assume that $H^{n-1}(H_\lambda,U_\l)=0$ for every $\lambda\in\L$. 
 Then there exists a map
 $$
 \iota^{n}\colon \bigoplus_{\lambda\in\Lambda} EH^{n}_b(H_\lambda,U_\lambda)\to EH^n(G,V)
 $$
 such that, for every $\alpha=(\alpha_\lambda)_{\lambda\in\L}\in \oplus_{\l\in\L}  EH^{n}_b(H_\lambda,U_\lambda)$,
$$
\res^n_\lambda(\iota^n(\alpha))=
i^n_\lambda(\alpha_\lambda)\quad {\rm for\ every}\ \lambda\in\Lambda
\ .$$
  \end{prop}
\begin{proof}
The definition of $\iota^n$ has already been described in the proof of Proposition~\ref{boundedcoc:prop}.
 Namely, once an element $\alpha=(\alpha_\lambda)_{\lambda\in\Lambda}\in \bigoplus_{\lambda\in\Lambda} EH^{n}_b(H_\lambda,U_\lambda)$
 is given, 
 for every $\lambda\in\Lambda$ we choose an alternating representative
$a_\lambda\in C^n_{b,\alt}(H_\lambda,U_\lambda)^{H_\lambda}$ of $\alpha_\lambda$, and an alternating
quasi-cocycle $\varphi_\lambda\in \QCa^{n-1}(H_\lambda,U_\l)^{H_\lambda}$ such that $a_\lambda=\delta^{n-1}\varphi_\lambda$.
Of course we may suppose that $\varphi_\lambda=0$ and $a_\lambda=0$ for all but a finite number of indices, so
$a=(a_\lambda)_{\lambda\in\Lambda}\in \oplus_{\lambda\in\Lambda} C^n_{b,\alt}(H_\lambda,U_\l)^{H_\lambda}$, 
$\varphi=(\varphi_\lambda)_{\lambda\in\Lambda}\in \oplus_{\lambda\in\Lambda} C^{n-1}_{b,\alt}(H_\lambda,U_\l)^{H_\lambda}$, and 
we can set $$\iota^n(\alpha)=[\delta^{n-1}\Theta^{n-1}(\varphi)]\ .$$ 

In order to prove that this definition of $\iota^n$ is well-posed, we need to show that, if $a_\lambda$ represents the null element
of $EH^n_b(H_\lambda,U_\l)$ for every $\lambda$, then $[\delta^{n-1}\Theta^{n-1}(\varphi)]=0$ in $EH^n_b(G,V)$. So, let us suppose that
$a_\lambda=\delta^{n-1} b_\lambda$ for some $b_\lambda\in C^{n-1}_b(H_\lambda,U_\l)^{H_\lambda}$, and set as usual
$b=(b_\lambda)_{\lambda\in\Lambda}\in \oplus_{\lambda\in\Lambda}C^{n-1}_b(H_\lambda,U_\l)^{H_\lambda}$. Up to replacing
$b_\lambda$ with $\alt^{n-1}_b(b_\lambda)$, we may suppose that $b_\lambda\in C^{n-1}_{b,\alt}(H_\lambda,U_\l)^{H_\lambda}$.
By construction we have $\delta^{n-1}(\varphi_\lambda-b_\lambda)=0$, so our assumption that
$H^{n-1}(H_\lambda,U_\l)=0$ implies that $\varphi=b+\delta^{n-2}c$ for some
$c\in \oplus_{\lambda\in\Lambda} C_{\alt}^{n-2}(H_\lambda,U_\l)$
(as usual, we may suppose that $c_\lambda$ is alternating).

Therefore, we have
\begin{align*}
\delta^{n-1}\Theta^{n-1}(\varphi)&=\delta^{n-1}\Theta^{n-1}(b)+\delta^{n-1}\Theta^{n-1}(\delta^{n-2}c)\\ &=
\delta^{n-1}\Theta^{n-1}(b)+\delta^{n-1}(\Theta^{n-1}(\delta^{n-2}c)-\delta^{n-2}\Theta^{n-2}(c))\ .
\end{align*}
By Theorem~\ref{maincc}, the right-hand side of this equality is the coboundary of a bounded cochain, and this concludes the proof.
\end{proof}

We will see in Proposition~\ref{secondocontro} that, if we drop the assumption that $H^{n-1}(H_\lambda,V)=0$ for every $\lambda\in\Lambda$, then the construction just described does not yield a well-defined map on exact bounded cohomology.

\medskip

The previous results may be exploited to deduce the non-vanishing of $EH^n_b(G,V)$ from 
the non-vanishing of $EH^n_b(H_\lambda,U_\lambda)$ for some $\lambda\in\Lambda$. 
For every group $K$ and every normed $K$-module $W$ we denote by $\overline{EH}^n_b(K,W)$ the quotient
of $EH^n_b(K,W)$ by the subspace of its elements with vanishing seminorm.

\begin{prop}\label{nonvanishing:general}
Suppose that, for every $\lambda$, there exists an $H_\lambda$-equivariant norm non-increasing retraction $V\to U_\lambda$. Then
\begin{align*}
\dim EH^n_b(G,V)& \geq \dim\left( \oplus_{\lambda\in\Lambda} EH^n_b(H_\lambda,U_\lambda)\right)\ ,\\
\dim \overline{EH}^n_b(G,V)& \geq \dim\left( \oplus_{\lambda\in\Lambda} \overline{EH}^n_b(H_\lambda,U_\lambda)\right)\ .
\end{align*}
\end{prop}
\begin{proof}
By Proposition~\ref{boundedcoc:prop}, the map $\oplus_{\lambda\in\Lambda} \res^n_\lambda$ establishes a bounded epimorphism
from $EH^n_b(G,V)$ to  
$$\left(\oplus_{\lambda\in\Lambda} i^n_\lambda\right)\, \left(\oplus_{\lambda\in\Lambda} EH^n_b(H_\lambda,U_\lambda)\right)\ 
\subseteq\ \oplus_{\lambda\in\Lambda} EH^n_b(H_\lambda,V)\ .$$
Therefore, in order to conclude it is sufficient to observe that the existence of an $H_\lambda$-equivariant retraction $V\to U_\lambda$
ensures that the map $i^n_\lambda \colon EH^n_b(H_\lambda,U_\lambda)\to EH^n_b(H_\lambda,V)$ is an isometric embedding.
\end{proof}

In order to obtain concrete non-vanishing results we exploit the following fundamental result about acylindrically hyperbolic groups.

\begin{thm}[Theorem 2.23 of~\cite{DGO}]\label{F2:thm}
Let $G$ be an acylindrically hyperbolic group. Then there exists a hyperbolically embedded subgroup $H$ of $G$ such that
$H$ is isomorphic to $F_2\times K$, where $K$ is finite.
\end{thm}

Putting together Proposition~\ref{nonvanishing:general} and Theorem~\ref{F2:thm} we may reduce
the non-vanishing of the exact bounded cohomology of an acylindrically hyperbolic group to the non-vanishing
of the cohomology of free non-abelian groups. As an application of this strategy we provide a proof
of Corollary~\ref{boundedco2:cor} stated in the introduction, which we recall here for convenience:

\begin{cor}\label{boundedco2t:cor}
 Let $G$ be an acylindrically hyperbolic group. Then the dimension of $\overline{EH}_b^3(G,\R)$
 is equal to the cardinality of the continuum. 
 \end{cor}
\begin{proof}
From Proposition~\ref{nonvanishing:general} and Theorem~\ref{F2:thm} we deduce
that 
$$\dim \overline{EH}^3_b(G,\R)\geq \dim \overline{EH}^3_b(F_2\times K,\R)=\dim \overline{H}^3_b(F_2\times K,\R)\ ,$$ where $K$ is a finite group. 
But $F_2\times K$ surjects onto $F_2$ with amenable kernel, so
${H}^3_b(F_2\times K,\R)$ is isometrically isomorphic to $H^3_b(F_2,\R)$, and
 the conclusion follows
from the fact that $\dim \overline{H}^3_b(F_2,\R)$ is equal to the cardinality of the continuum~\cite{Soma3}.
\end{proof}

Monod and Shalom  showed the importance of bounded
cohomology with coefficients in $\ell^2(G)$ 
in the study of rigidity of $G$~\cite{MS1,MS2,MS3}, and proposed the condition $H^2_b(G,\ell^2(G))\neq 0$ 
as a cohomological definition of negative curvature for groups. More in general, bounded cohomology 
with coefficients in $\ell^p(G)$, $1\leq p<\infty$ has been widely studied as a powerful tool to prove (super)rigidity
results (see e.g.~\cite{Ha-isomhyp} and~\cite{CFI}). However, little is known in this context about degrees higher than two. 
The following result shows that the non-vanishing of $H^n_b(G,\ell^p(G))$ may be reduced to
the non-vanishing of $H^n_b(F_2,\ell^p(F_2))$ for a wide class of groups. However, as far as the authors know,
for no degree $n\geq 3$ it is known whether $H^n_b(F_2,\ell^p(F_2))$ vanishes or not.

\begin{cor}\label{lp:cor}
Let $G$ be an acylindrically hyperbolic group. Then, for every $p\in [1,\infty)$, $n\geq 2$, we have
$$
\dim EH^n_b(G,\ell^p(G))\geq \dim H^n_b(F_2,\ell^p(F_2))\ .
$$
\end{cor}
\begin{proof}
By Theorem~\ref{F2:thm} there exists a hyperbolically embedded subgroup $H$ of $G$
which is isomorphic to a product $F_2\times K$, where $K$ is a finite group.
We identify $\ell^p(H)$ with the $H$-submodule of
$\ell^p(G)$ given by those functions that vanish outside $H$. Then, it is immediate to realize that
$\ell^p(G)$ admits an $H$-equivariant norm non-increasing retraction onto $\ell^p(H)$. 
By
Proposition~\ref{nonvanishing:general}, this implies that
 $$\dim EH^n_b(G,\ell^p(G))\geq \dim EH^n_b(H,\ell^p(H))\ .$$
 The conclusion follows from the fact that, since $F_2$ is a retract of $H\cong F_2\times K$ and
 $\ell^p(H)$ admits an $F_2$-equivariant norm non-increasing retraction onto $\ell^p(F_2)$, 
there exist isometric embeddings
$$
H^n_b(F_2,\ell^p(F_2))=EH^n_b(F_2,\ell^p(F_2))\to EH^n_b(F_2,\ell^p(H))\to EH^n_b(H,\ell^p(H))\ .
$$
\end{proof}

Corollary \ref{boundedco2t:cor} can be combined with Theorem \ref{dgo+} that we will prove in the appendix to show the non-injectivity of the restriction map we announced in the introduction, whose statement we recall here for the sake of completeness:
\begin{prop}
For any non-degenerate hyperbolically embedded finite family of subgroups $\{H_\l\}_{\l\in\L}$ of a finitely generated group $G$ and for any normed $\R[G]$-module $V$, the kernel of the restriction map $\oplus _{\lambda\in\Lambda}\res^3_\lambda\colon H^3_b(G,V)\to \bigoplus_{\lambda\in \Lambda} H^3_b(H_\lambda,V)$ 
is infinite-dimensional.
\end{prop}
\begin{proof}
Of course, it is sufficient to show that the kernel of the restriction map $\oplus _{\lambda\in\Lambda}\res^3_\lambda\colon EH^3_b(G,V)\to \bigoplus_{\lambda\in \Lambda} EH^3_b(H_\lambda,V)$ is infinite-dimensional.

Since $G$ is finitely generated we can assume
that the family $\{H_\l\}_{\l\in\L}$ is hyperbolically embedded in $(G,X)$ where $X$ is a generating system for $G$ (see~\cite[Corollary 4.27]{DGO}). 
In particular Theorem \ref{dgo+} implies that there exists a free group on two generators $F$ such that the family $\{H_\lambda\}_{\l\in\L}\cup \{F\times K(G)\}$ is hyperbolically embedded in $(G,X)$.
Also observe that, since $\Lambda$ is finite, the direct sum of restrictions define maps
\begin{align*}
\widehat{\eta}&\colon EH^3_b(G,V)\to \bigoplus_{\lambda\in \Lambda} EH^3_b(H_\lambda,V)\oplus EH^3_b(F\times K(G))\, ,\\
{\eta}& \colon EH^3_b(G,V)\to \bigoplus_{\lambda\in \Lambda} EH^3_b(H_\lambda,V) .
\end{align*}
 These maps fit in the following commutative diagram
$$\xymatrix{EH^3_b(G,V)\ar[r]^{\eta}\ar[dr]_{\widehat{\eta}}& \bigoplus_{\lambda\in \Lambda} EH^3_b(H_\lambda,V)\\&\bigoplus_{\lambda\in \Lambda} EH^3_b(H_\lambda,V)\oplus H^3_b(F\times K(G))\ar[u]}$$
where the vertical arrow is the projection onto the first summand.

But
the module $EH^3_b(F\times K(G))$ is infinite dimensional, and the map $\widehat{\eta}$ is surjective by Proposition \ref{boundedcoc:prop},
so the conclusion follows.
\end{proof}

\section{Examples and counterexamples}\label{counterexample:sec}
In this section we prove Propositions~\ref{1contro:prop} and~\ref{prop:1}, and we provide examples showing that, in general, the map $\Theta^\bullet$
constructed in Theorem~\ref{maincc} does not induce a well-defined map on exact bounded cohomology.
Throughout the whole section we will exploit the well-known fact that, if $X$ is an aspherical manifold, then
the ordinary and the bounded cohomology of $X$ are canonically isomorphic to the ones of $\pi_1(X)$.

We begin with the following:

\begin{lemma}\label{toplemma}
Let $M$ be a compact orientable $(n+1)$-dimensional manifold with connected boundary, and suppose that the following conditions hold:
\begin{itemize}
\item $M$ and $\partial M$ are aspherical;
\item the inclusion $\partial M\to M$ induces an injective map on fundamental groups;
\item $\pi_1(\partial M)$ is Gromov hyperbolic.
\end{itemize}
Then the restriction map
$$
\res^n_b\colon H^n_b(\pi_1(M),\R)\to H^n_b(\pi_1(\partial M),\R)
$$
is not surjective.
\end{lemma}
\begin{proof}
Let us consider the commutative diagram 
$$
\xymatrix{
H^{n}_b(\pi_1(M),\R)\ar[r]^{{\res}^n_b} \ar[d]^{c_M^{n}} & H^{n}_b(\pi_1(\partial M),\R)  \ar[d]^{c_{\partial M}^{n}} \\
H^{n}(\pi_1(M),\R)\ar[r]^{\res^n} & H^{n}(\pi_1(\partial M),\R)\ ,
}
$$
where vertical arrows represent comparison maps. 
Since
any cycle in $Z^n(\partial M,\R)$ bounds in $M$, by the Universal Coefficient Theorem the restriction of any element in $H^n(M,\R)$
to $H^n(\partial M,\R)$ is null.
Since restrictions commute with the canonical isomorphisms $H^n(M,\R)\cong H^n(\pi_1(M),\R)$,
$H^n(\partial M,\R)\cong H^n(\pi_1(\partial M),\R)$,
this implies that $\res^n$ is the zero map. But $H^n(\pi_1(M),\R)\cong H^n(\partial M,\R)\cong\R\neq 0$, so
the composition $c^n_{\partial M}\circ \res^n_b$ cannot be surjective. Now the main result of~\cite{Min}
implies that $c^n_{\partial M}$ is an epimorphism, so we can conclude that $\res^n_b$ cannot be surjective.
\end{proof}

We are now ready to prove Proposition~\ref{1contro:prop} from the introduction:

\begin{prop}\label{prop:7.2}
 For every $n\geq 2$, there exists a pair $(G,H)$ such that $G$ is relatively hyperbolic with respect to $H$
(in particular, $H$ is hyperbolically embedded in $G$), and the restriction $H^n_b(G,\R)\to H^n_b(H,\R)$ is not surjective.
 \end{prop}
\begin{proof}
By~\cite{Reid}, for every $n\geq 2$ there exist examples of compact orientable $(n+1)$-dimensional hyperbolic manifolds with connected
geodesic boundary. Let $M^{n+1}$ be one such example, and let us set $G=\pi_1(M)$, $H=\pi_1(\partial M)$. 
It is well-known that $G$ is relatively hyperbolic with respect to $H$.
Moreover, the manifold $M^{n+1}$
satisfies all the conditions described in Lemma~\ref{toplemma}, so the restriction $H^n_b(G,\R)\to H^n_b(H,\R)$ is not surjective.
\end{proof}
 


We now provide examples where the map $\Theta^\bullet$ defined in Theorem~\ref{maincc} does not induce a well-defined map in bounded cohomology.

\begin{prop}\label{secondocontro:pre}
Let $\{H_\lambda\}_{\lambda\in\Lambda}$ be a  family of subgroups of the group $G$, and
denote by $j^\lambda_{n}\colon H_n(H_\lambda,\R)\to H_n(G,\R)$ the map induced by the inclusion $H_\lambda\to G$.
Suppose that the following
conditions hold:
\begin{itemize}
\item $H_\lambda$ is amenable for every $\lambda\in\Lambda$;
\item The map $\oplus_{\lambda\in\Lambda} j_n^\lambda\colon \oplus_{\lambda\in\Lambda}H_n(H_\lambda,\R)\to H_n(G,\R)$ is not injective.
\end{itemize}
Then there exists a collection $\varphi=(\varphi_\lambda)_{\lambda\in\Lambda}\in \oplus_{\lambda\in\Lambda}
Z^n(H_\lambda,\R)^{H_\l}$ of genuine cocycles such that,  if $\Phi\in \QC^n(G,\R)^G$ is any quasi-cocycle such that
$\Phi|_{H_\lambda}$ stays at uniformly bounded distance from $\varphi_\lambda$ for every $\lambda\in\Lambda$, then
$[\delta^n\Phi]\neq 0$ in $H^{n+1}_b(G,\R)$.
\end{prop}
\begin{proof}
Let $(w_\lambda)_{\lambda\in\Lambda}$ be a non-null element of $\ker \left(\oplus_{\lambda\in\Lambda} j_n^\lambda\right)$.
Since $w_\lambda\neq 0$ for some $\lambda\in\Lambda$, 
by the Universal Coefficient Theorem we may choose  cocycles
$\varphi_\lambda\in Z^n(H_\lambda,\R)^{H_\lambda}$ in such a way that $\varphi_\l=0$ for all but a finite number of indices and, if $z_\lambda$ is any representative
of $w_\lambda$, then  $\sum_{\lambda\in\Lambda}\varphi_\lambda(z_\lambda)=1$.

Suppose now that $\Phi$ is as in the statement.
Then for every $\lambda\in\Lambda$ there exists $b_\lambda\in C^n_b(H_\lambda,\R)^{H_\lambda}$ 
such that $\Phi|_{H_\lambda}=\varphi_\lambda+b_\lambda$. We set
$M=\sup_{\lambda\in\Lambda} \|b_\lambda\|_\infty$.

Let us assume by contradiction that  $[\delta^n \Phi]$ vanishes in $H^{n+1}_b(G,\R)$. 
This implies that  $\Phi=\psi+b$, where $\psi\in Z^n(G,\R)$ and $b\in C^n_b(G,\R)$. 
Since each $H_\lambda$ is amenable, the $\ell^1$-seminorm on $H_n(H_\lambda,\R)$ vanishes (see e.g.~\cite{Matsu-Mor}), so
for every $\lambda\in\Lambda$ we can choose a representative $z_\lambda\in Z_n(H_\lambda,\R)$ of $w_\lambda$
such that $\sum_{\lambda\in\Lambda} \|z_\l\|_1<(M+\|b\|_\infty)^{-1}$.

Let us set $z=\sum_{\lambda\in\Lambda} z_\lambda$.
Since $\sum_{\lambda\in\Lambda} i_n^\lambda(w_\lambda)=0$,
we have $\psi(z)=0$, and 
\begin{equation}\label{prima:eq}
|\Phi(z)|=|b(z)|\leq  \|b\|_\infty\|z\|_1<\frac{\|b\|_\infty}{M+\|b\|_\infty}\ .
\end{equation}
On the other hand, we have
$$
\Phi(z)=\sum_{\lambda\in\Lambda} \Phi(z_\lambda)=\sum_{\lambda\in\Lambda} \varphi_\lambda(z_\lambda)+
\sum_{\lambda\in\Lambda} b_\lambda(z_\lambda)=1+\sum_{\lambda\in\Lambda} b_\lambda(z_\lambda)\ ,
$$
so
$$
|\Phi(z)|\geq 1-\left|\sum_{\lambda\in\Lambda} b_\lambda(z_\lambda)\right|\geq1-M\sum_{\lambda\in\Lambda} \|z_\lambda\|_1
> 1-\frac{M}{M+\|b\|_\infty} \ .
$$
This contradicts inequality~\eqref{prima:eq}, and concludes the proof.
\end{proof}

Together with our main result on extensions of quasi-cocycles, Proposition \ref{secondocontro:pre} readily implies the following:

\begin{cor}\label{nonv:cor}
Let $\{H_\lambda\}_{\lambda\in\Lambda}$ be a finite hyperbolically embedded family of subgroups of the group $G$, and
denote by $j^\lambda_{n}\colon H_n(H_\lambda,\R)\to H_n(G,\R)$ the map induced by the inclusion $H_\lambda\to G$.
Suppose that the following
conditions hold:
\begin{itemize}
\item $H_\lambda$ is amenable for every $\lambda\in\Lambda$;
\item The map $\oplus_{\lambda\in\Lambda} j_n^\lambda\colon \oplus_{\lambda\in\Lambda}H_n(H_\lambda,\R)\to H_n(G,\R)$ is not injective.
\end{itemize}
Then $H^{n+1}_b(G,\R)\neq 0$.
\end{cor}

The following proposition provides concrete examples for the phenomenon described in Proposition~\ref{secondocontro:pre}.

\begin{prop}\label{secondocontro}
For every $n\geq 1$, there exist a group $G$ relatively hyperbolic
with respect to the finite family of subgroups $\{H_\lambda\}_{\lambda\in\Lambda}$
such that  the following holds. 
There exists a collection $\varphi=(\varphi_\lambda)_{\lambda\in\Lambda}\in \oplus_{\lambda\in\Lambda}
Z^n_{\alt}(H_\lambda,\R)$ of genuine cocycles such that, if $\Phi=\Theta^n(\varphi)\in \QCa^n(G,\R)$ is the quasi-cocycle constructed
in Theorem~\ref{maincc}, then  
$[\delta^n\Phi]\neq 0$ in $H^{n+1}_b(G,\R)$.
\end{prop}
\begin{proof}
Let $M$ be an orientable complete finite-volume non-compact hyperbolic $(n+1)$-manifold, and let $\Lambda$ be the set
of cusps of $M$. We set $G=\pi_1(M)$, and we denote by $H_\lambda$ the subgroup of $G$ corresponding
to the cusp of $M$ indexed by $\lambda$. 
It is well-known that $\Lambda$ is finite, that each cusp is $\pi_1$-injective in $M$, and that
$G$ is relatively hyperbolic with respect to
 $\{H_\lambda\}_{\lambda\in\Lambda}$, so in order to conclude it is sufficient to show that
 $(G,\{H_\lambda\}_{\lambda\in\Lambda})$ satisfies the hypotheses of Proposition~\ref{secondocontro:pre}.

Each $H_\lambda$ is the fundamental group of a compact Euclidean $n$-manifold, so it is virtually abelian, hence amenable.
Moreover, if $C_1,\ldots,C_k$ are the cusps of $M$, then the map $\oplus_{i=1}^k H_{n}(C_i,\R)\to H_n(M,\R)$
is not injective. Since the spaces $M$ and $C_i$, $i=1,\ldots,k$, are all aspherical, this implies in turn that
the map $\oplus_{\lambda\in\Lambda} j_n^\lambda\colon \oplus_{\lambda\in\Lambda}H_n(H_\lambda,\R)\to H_n(G,\R)$ is not injective.
\end{proof}

We can also use manifolds constructed by \cite{MRS-collisions} to show the following.

\begin{prop}
 \label{interm}
 For every $d\geq 4$ and $2\leq k\leq d$ there exist infinitely many commensurability classes of cusped orientable 
 hyperbolic $d$-manifolds $M$ such that $H^k_b(\pi_1(M),\R)$ is non-trivial.
\end{prop}

\begin{proof}
Let us first suppose that $2<k<d$.
 The authors of \cite{MRS-collisions} construct  infinitely many commensurability classes of cusped orientable hyperbolic $d$-manifolds $M$ containing a 
 properly embedded totally geodesic submanifold $N$ of dimension $k$ with the following property. 
 Denoting by $\mathcal E$ the cusp cross-sections of $M$ and by $\mathcal F$ the cusp cross-sections of $N$, we have that:
 \begin{enumerate}
  \item
   the homomorphism $H_{k-1}(\mathcal F)\to H_{k-1}(\mathcal E)$ induced by the inclusion $N\to M$ is an injection.
  \item
   the homomorphism $H_{k-1}(\mathcal F)\to H_{k-1}(M)$ induced by the inclusion $N\to M$ is \emph{not} an injection. 
 \end{enumerate}
Denoting by $\{H_\l\}_{\l\in\L}$ the set of all fundamental groups of the cusps of $M$ and $G=\pi_1(M)$, we then see that the conditions of Corollary \ref{nonv:cor} 
are satisfied for $n=k-1$. In particular, $H^k_b(\pi_1(M),\R)\neq 0$, as required. Let us now consider the cases $k=2,d$.
Since $\pi_1(M)$ is relatively hyperbolic we have $H^2_b(\pi_1(M),\R)\neq 0$. Moreover, it is well-known that the straightened volume form on $M$ (i.e.~the
$d$-dimensional cochain obtained by integrating the volume form on straight simplices) defines a non-trivial coclass
in $H^d_b(M,\R)\cong H^d_b(\pi_1(M),\R)$, and this concludes the proof.
\end{proof}

As mentioned at the end of the proof of the previous proposition, for every cusped orientable hyperbolic
$n$-manifold $M$, the straightened volume form defines a non-trivial volume coclass  $\omega_M\in H^n_b(\pi_1(M),\R)$. We pose here the following:

\begin{quest}
Let $\overline{M}$ denote the compact manifold with boundary obtained by truncating the cusps of $M$ along horospherical
sections of the cusps, denote by $H_\lambda$, $\lambda\in\Lambda$, the subgroups of $\pi_1(M)$ corresponding to the
fundamental groups of the boundary components
of $\overline{M}$, and let
 $\alpha$ be the element of $\oplus_{\lambda\in\L} C^{n-1}(H_\l,\R)$ corresponding to the (unbounded) Euclidean volume form on
 $\partial\overline M$. Is it true that $[\delta^{n-1}\Theta^{n-1}(\alpha)]=\omega_M$ in $H^n_b(\pi_1(M),\R)$?
\end{quest}

\section{Extension of cocycles and projections}\label{sec:reconsproj}

The purpose of this section is to give a precise meaning to the Informal Statement from the introduction, which says that, whenever $H$ is a hyperbolically embedded
subgroup of a finitely generated group $G$, our extension to $G$ of a specific cocycle on $H$ encodes the geometry of the embedding of $H$ in $G$. 
In fact, we will construct a cocycle whose extension to $G$ will allow us to define a family of projections $\pi_{gH}:G\to gH$ that satisfy the BBF axioms by evaluating the extension on suitable tuples, see Proposition \ref{prop:generalexten} and Proposition \ref{prop:Zexten}. Recall from Theorem \ref{BBF:thm} that this is enough information to know that $H$ is hyperbolically embedded in $G$.

The cocycle we will construct lies in $Z^2(H,\ell^2(E_H))$, where $E_H$ is the set of edges of a Cayley graph of $H$, and is in fact an unbounded modification of the cocycle studied by Monod, Mineyev and Shalom 
in~\cite{MMS}. In the case when $H\cong \mathbb{Z}^n$, $n\geq 2$, we can replace this cocycle by an $n$-cocycle with trivial coefficients: namely,
the volume cocycle in $Z^n(\mathbb{Z}^n,\R)$.


\subsection{Construction for arbitrary groups}
Let us fix a hyperbolically embedded subgroup $H$ of the group $G$. In order to avoid trivialities,
we assume that $H$ is proper and infinite. Moreover, throughout the section we assume that $G$ is finitely generated.
As a consequence, $H$ is also finitely generated, and we can choose a symmetric finite generating set $\calS_H$ of $H$
contained in a symmetric finite generating set $\calS_G$ of $G$. In this way, the Cayley graph $\Cay(H,\calS_H)$ is naturally a subgraph of $\Cay(G,\calS_G)$. 
By~\cite[Corollary 4.27]{DGO}, we can fix a choice of $X\subseteq G$ so that $H\hookrightarrow_h (G,X)$ and $X$ contains $S_H$.
As usual, we denote by $d_H$ the relative metric on $H$ (see Definition~\ref{defhypemb}). As observed in Remark~\ref{finitelygen:rem},  
the metric $d_H$ is bi-Lipschitz equivalent to the word metric $d_{\calS_H}$ of $\Cay(H,\calS_H)$.

We will denote by $E_G$ (resp. $E_H$) the set of oriented edges of $\Cay(G,\calS_G)$ 
(resp. of $\Cay(H,\calS_H)$) and we will consider the normed $G$-module $\ell^2(E_G)=\ell^2(E_G,\R)$ together with its $H$-submodule $\ell^2(E_H)=\ell^2(E_H,\R)$
(where we identify an element $f\in \ell^2(E_H)$ with the function in $\ell^2(E_G)$ which coincides with $f$ on $E_H$
and vanishes elsewhere).
We will also consider the bounded operator
$$\begin{array}{cccc}
   \Psi_H\colon &\ell^2(E_H)&\to&\ell^2(H)
  \end{array}
$$
that associates to any function $f$ in $\ell^2(E_H)$ the function $\Psi_H (f) \in \ell^2(H)$ defined by
$$\Psi_H( f)(g)=\sum_{t(e)=v}f(e)-\sum_{o(e)=v}f(e) $$
where $t$ (resp. $o$) is the function  $E_H\to H$ associating to the edge $e$ its final point (resp. its starting point).
In the very same way, one can define the bounded operator
$$\begin{array}{cccc}
   \Psi_G\colon &\ell^2(E_G)&\to&\ell^2(G)\ .
  \end{array}
$$

We fix elements $h_n$ of $H$ such that $d_{\calS_H}(1,h_n)=n^3$ (such elements exist since $H$ is infinite).

\begin{lemma}\label{lem:1}
 For any point $z$ in $H$ there exists a constant $c$ such that
 \begin{enumerate}
  \item $n^3-c\leq d_{\calS_H}(h_n,z)\leq n^3+c$ 
  \item $d_{\calS_H}(h_{n+1},z)-d_{\calS_H}(h_n,z)\geq 3n^2-2c$.
 \end{enumerate}
\end{lemma}
\begin{proof}
Set $c=d_{\calS_H}(1,z)$. Then (1) follows from the triangle inequality, and (2)
follows from (1) and (another) triangle inequality.
\end{proof}

Let us now consider the cochain $c_H\in C^1_{\alt}(H,\ell^2(E_H))^H$ defined by

$$c_H(l_0,l_1)=\frac{d_{\calS_H}(l_0,l_1)}{2\#[l_0,l_1]}\sum_{\g\in[l_0,l_1]}\chi_\g-\chi_{\ov\g}.$$
Here we denote by $[l_0,l_1]$ the set of geodesic paths in $\Cay(H,\calS_H)$ with endpoints $l_0,l_1$, 
if $\g$ is a geodesic path  we denote by $\ov\g$ the geodesic with the opposite orientation, and, given a geodesic path $\g$, 
we denote by $\chi_\g$ the function in $\ell^2(E_H)$ that takes value 1 on the oriented edges that are contained in $\g$ and is null everywhere else.

The element $\delta^1 c_H$ is an unbounded cocycle in $C^2_{{\rm alt}}(H,\ell^2(E_H))^H$ that separates the points of $H$ in the sense specified by the following lemma:
\begin{lemma}\label{lem:2}
 For every $z\in H$ we have
 $$
 \lim_{n\to \infty} \Psi_H(\delta^1 c_H(h_n,h_{n+1},z))(z)=\infty \ .$$
 Moreover $z,h_n,h_{n+1}$ are the only elements of $H$ on which $\Psi_H(\delta^1 c_H(h_n,h_{n+1},z))$ can be nonzero. 
\end{lemma}

\begin{proof}
Let $h,h',y$ be elements of $H$. It follows from the definition of $c_H$ that $\Psi_H(c_H(h,h'))(y)\neq 0$ only if $y$ 
 belongs to a geodesic between $h$ and $h'$. Moreover, if $y\neq h$, $y\neq h'$, then any such geodesic 
 has an edge pointing to $y$ and an edge exiting from $y$, so its contribution
 to $\Psi_H(c_H(h,h'))(y)$ is null. This implies the second assertion of the Lemma: $\Psi_H(\delta^1c_H(h_n,h_{n+1},z))(x)=0$ if $x\notin \{z,h_n,h_{n+1}\}$.

 On the contrary, it is immediate to check that 
 $$
 \Psi_H(c_H(h,h'))(h)=-d_{\calS_H}(h,h') \, ,\quad  \Psi_H(c_H(h,h'))(h')=d_{\calS_H}(h,h')\ .
 $$
 This justifies the first assertion in the statement: if $n$ is sufficiently large, then $z\neq h_n$ and $z\neq h_{n+1}$, so
 $$
 \Psi_H(\delta^1 c_H(h_n,h_{n+1},z))(z)=d_{\calS_H}(h_{n+1},z)-d_{\calS_H}(h_n,z)\ ,
 $$ 
and the conclusion follows from
Lemma \ref{lem:1}.
\end{proof}

Let now $C_G$ denote the quasi-cocycle $\Theta^2(\delta^1 c_H)\in C^2_{\alt}(G,\ell^2(E_G))^G$ that extends the cocycle $\delta^1 c_H$,
where $\Theta^2$ is the extension operator provided by Theorem~\ref{maincc}.
The quasi-cocycle $C_G$ allows us to reconstruct the projections $\pi_{gH}$:
indeed let us define, for any coset $gH$ of $H$ and for any point $y$ in $\widehat{G}$, the projection $\ov \pi_{gH}(y)$ to be  the set of points $\ov z\in gH$  
such that
$$
\lim_{n\to \infty} \Psi_G(C_G(gh_n,gh_{n+1},y))(\ov z)=\infty\ .
$$

\begin{prop}\label{prop:generalexten}
 The projections $\ov\pi_{gH}$ are well defined and coincide with $\pi_{gH}$ as in Definition \ref{projdefn}. In particular, they satisfy the BBF axioms. 
\end{prop}
\begin{proof}
 Fix $y\in G$, and
 let us consider the simplex $\ov s_n(y)=(gh_n,gh_{n+1},y)$. 
 We first prove that, if $n$ is large enough, then the only coset in $\calR(\ov s_n(y))$ on which the trace 
 of $\overline{s}_n(y)$ is not degenerate is $gH$.
 In fact, if $n$ is large then $d_{gH}(gh_n,gh_{n+1})>2D$, $d_{gH}(gh_n,\pi_{gH}(y))>2D$, and also $d_{gH}(gh_{n+1},\pi_{gH}(y))>2D$ 
 (here $D$ is as usual the constant provided by Lemma \ref{farproj}): indeed $d_{gH}(gh_n,gh_{n+1})=d_{H}(h_n,h_{n+1})$ and $d_H$ is bi-Lipschitz equivalent to the word metric $d_{\calS_H}$. 
 
 This implies that $gH$ is in $\calR(\ov s_n(y))$.
 Moreover, every geodesic joining $y$ with $gh_n$ or with $gh_{n+1}$ must contain $c(gH)$. This readily implies that,
 if $B\in\calR(\ov s_n(y))\setminus \{gH\}$, then $\pi_B(gh_n)=\pi_B(gh_{n+1})=\pi_B(c(gH))$. Therefore, for any such $B$
 the trace $\tr_2^B(\ov s_n(y))$ is degenerate.
 
 Therefore, as a consequence of the definition of $\Theta^2$, if $n$ is large enough, then the function $C_G(gh_n,gh_{n+1},y)$ 
 is supported on $g\cdot E_{H}$
and 
 $$C_G(gh_n,gh_{n+1},y)=g\cdot\left(\frac 1{|\pi_{gH}(y)|}\sum_{gh\in\pi_{gH}(y)}\delta^1 c_H(h_n,h_{n+1},h)\right)\ .$$

In particular it follows from Lemma \ref{lem:2} that, as $n$ tends to infinity, the quantity
$$\Psi_G(C_G(gh_n,gh_{n+1},y))(\ov z)= \frac 1{|\pi_{gH}(y)|}\sum_{gh\in\pi_{gH}(y)} \Psi_H(\delta^1 c_H(h_n,h_{n+1},h))(g^{-1}\ov z)$$
tends to infinity if and only if $\ov z$ is in $\pi_{gH}(y)$. In particular, $\pi_{gH}$ and $\ov \pi_{gH}$ coincide.
\end{proof}

As stated in the introduction,
the construction described in Proposition \ref{prop:generalexten} is based on the use of higher degree quasi-cocycles: 
 in order to be able to recover the projection of a point $z$ on $H$, 
we make use of (a sequence of) two auxiliary extra points, that in our construction are 
 provided by the sequence of pairs $(h_n,h_{n+1})$. Moreover, here we have used 
 the general formulation of our extension theorem: 
 the cocycle $c_H$ takes values in a proper submodule of the coefficient module of its extension to $G$.

\subsection*{The case $\Z^n\hookrightarrow_h G$, with $n>1$}
In the case
when  $G$  admits $\Z^n$ as a hyperbolically embedded subgroup, 
we are able to reconstruct the projections on the cosets of $\Z^n$ from the 
extension of a certain alternating $n$-cocycle with \emph{real} coefficients.

Let us consider the inclusion of $\Z^n$ in $\R^n$ and, for any $(n+1)$-tuple $\ov z=(z_0,\ldots,z_n)$ of elements in $\Z^n$, let us denote by $\Delta(\ov z)$ the affine simplex with vertices $(z_0,\ldots,z_n)$.
Let us moreover define the cocycle $\alpha\in C^n(\Z^n,\R)$ by prescribing that the value of $\alpha(\ov z)$ is the signed Euclidean volume of the simplex $\Delta (\ov z)$.
For every integer $m$, we will denote by $y_0^m$ the point $me_1$ and by $y_i^m$ the point $me_1+me_i$, where $e_1,\ldots, e_n$ are the natural generators of $\Z^n$.

For $z\in G$ (in particular $z$ might be an element of $\Z^n$) denote by $\ov s^m_i(z)$ the simplex $(y_0^m,\ldots, y_{i-1}^m,z,y_{i+1}^m,\ldots, y_{n}^m)\in G^{n+1}$.

\begin{lemma}\label{lem:0}
 The cocycle $\alpha\in C^n(\Z^n,\R)$ is alternating.  Moreover, for any $z\in \Z^n$ and for any $m$ in $\N$, the $i$-th coordinate of $z$ can be computed for $i>1$ as $$\frac {n!} {m^{n-1}}\alpha(\ov s^m_i(z))$$
 and for $i=1$ as
  $$m+\frac {n!} {m^{n-1}}\alpha(\ov s^m_1(z)).$$
\end{lemma}
\begin{proof}
 The signed area of the Euclidean simplex with vertices $(y_0,\ldots,y_n)$ can be computed using the determinant of the matrix whose columns are the coordinates of 
 $y_j-y_0$. In particular for $i=1$ we get
 
 $$\alpha(\ov s^m_1(z))=\frac 1{n!}\det \begin{pmatrix}z_1-m&0\\z_2&m\\\vdots&&\ddots\end{pmatrix}=\frac{m^{n-1}}{n!}(z_1-m).$$
 Analogously one gets $\alpha(\ov s^m_i(z))=\frac{m^{n-1}}{n!}z_i$.
\end{proof}

Let now $G$ be a group with $\Z^n\hookrightarrow_h G$, and let $A\in\QCa(G,\R)$ be the quasi-cocycle obtained by setting $A=\Theta^n (\alpha)$,
where $\Theta^n$ is the map described in Theorem~\ref{maincc}.

Let us define, for any coset $g\Z^n$ and any point $z\in G$, the projection $\ov\pi_{g\Z^n}(z)$ to be the point 
$$ \ov\pi_{g\Z^n}(z)=g\lim_{m\to \infty}\left(\left\lfloor m+\frac {n!A(\ov s^m_1(g^{-1}z))} {m^{n-1}}\right\rfloor e_1+\sum_{i=2}^n \left\lfloor\frac {n!A(\ov s^m_i(g^{-1}z))} {m^{n-1}}\right\rfloor e_i\right).$$

\begin{prop}\label{prop:Zexten}
 The projections $\ov\pi_{g\Z^n}$ are well defined and coincide up to bounded error with $\pi_{g\Z^n}$ as in Definition \ref{projdefn}. In particular, they satisfy the BBF axioms. 
\end{prop}
\begin{proof}
We will show that the expression in the definition of  $\ov\pi_{g\Z^n}(z)$ is eventually constant in $m$, in particular $\ov\pi_{g\Z^n}$ is well defined. 
In showing this, we will prove that $\ov\pi_{g\Z^n}$ is at bounded distance from $\pi_{g\Z^n}$, hence concluding the proof.

Let $D$ be the constant given by Lemma \ref{farproj}. 
For every $i$ let us consider the simplex $\ov s^m_i(g^{-1}z)$. We claim that, if $m$ is large enough, then $\Z^n$ is the only coset in 
$\mathcal R(\ov  s^m_i(g^{-1}z))$ on which the trace of $\ov  s^m_i(g^{-1}z)$ is not degenerate.
 
In fact, for every sufficiently large $m$, 
we have that $d_{\mathbb{Z}^n}(y_k^m,\pi_{\Z^n}(g^{-1}z))>D$ for every $k=0,\ldots,n$.
Therefore, Lemma \ref{farproj} implies that any geodesic with endpoints $g^{-1}z$ and $y_k^m$ contains the point $c(\Z^n)$. 
Assume by contradiction that there exists a coset $B$ in $\mathcal R(\ov s_i^m(g^{-1}z))$ different from $\Z^n$ on which the trace $\tr_n^B(\ov s_i^m(g^{-1}z))$ is not degenerate.
Then $B$ belongs to $\mathcal S(g^{-1}z,y_k^m)$ for some $k\neq i$, and hence also to $\mathcal S(g^{-1}z,c(\Z^n))$.  But this 
readily implies that, for every $j\neq i$, the projection $\pi_B(y_j^m)$ coincides with $\pi_B(c(\Z^n))$. This implies that the trace $\tr_n^B(\ov s_i^m(g^{-1}z))$ is degenerate, which is a contradiction.

Since $\Z^n$ is the only coset in $\mathcal R (\ov s^m_i(g^{-1}z))$ and $ \ov s^m_i(g^{-1}z)$ is not small, 
by definition of $\Theta^n$ we have that $ A(\ov  s^m_1(g^{-1}z))=\tr_n^{\Z^n}(\ov  s^m_1(g^{-1}z))$.
Therefore, 
if we denote by $h_i$ the $i$-th coordinate of $h\in \Z^n$, then by Lemma~\ref{lem:0} we get
 $$ A(\ov  s^m_1(g^{-1}z))=\frac1{|\pi_{\Z^n}(g^{-1}z)|}\sum_{h\in\pi_{\Z^n}(g^{-1}z)}\frac{m^{n-1}}{n!} (h_1-m)$$
and
$$ A(\ov  s^m_i(g^{-1}z))=\alpha(\tr_n^{\Z^n}(\ov  s^m_i(g^{-1}z)))=\frac1{|\pi_{\Z^n}(g^{-1}z)|}\sum_{h\in\pi_{\Z^n}(g^{-1}z)}\frac{m^{n-1}}{n!} h_i$$
if $i>1$.
This implies that the expression defining $\ov\pi_{g\Z^n}$ is eventually constant in $m$, hence the limit is well defined.

Since the projection of $g^{-1}z$ to the coset $\Z^n$ has diameter at most $D$, and a consequence of what we just proved is that the expression 
for $\ov\pi_{g\Z^n}$ is a convex combination of the points in $\pi_{\Z^n}(g^{-1}z)$, we get that the point 
$\ov \pi_{g\Z^n}(z)=g\ov \pi_{\Z^n}(g^{-1}z)$ has distance at most $D$ from the set $\pi_{g\Z^n}(z)=g\pi_{\Z^{n}}(g^{-1}z)$. 
\end{proof}



\appendix
\section{On virtually free hyperbolically embedded subgroups}

In most cases, the ``natural'' hyperbolically embedded subgroups one finds in a given group are virtually cyclic 
(for example this is the case for mapping class groups). However, 
many applications are based on the existence of virtually free non-abelian hyperbolically embedded subgroups. 
A crucial part of \cite{DGO} is therefore devoted to show that virtually free hyperbolically embedded subgroups can be constructed starting from a non-trivial 
hyperbolically embedded family. Theorem \ref{dgo+} provides a strengthened version of this construction, that we used in the proof of Proposition~\ref{infinitekernel:prop}.

 Let $G$ be a group containing a non-degenerate hyperbolically embedded 
 family of subgroups. 
 As observed before the statement of Theorem~\ref{thm:DGO+e}, 
 $G$ contains a maximal finite normal subgroup, which will be denoted $K(G)$.

\begin{thm}[\cite{DGO}, Theorem 6.14$+\epsilon$]
\label{dgo+}
 Let $X$ be a (possibly infinite) generating system of the group $G$ and let 
 the non-degenerate family of subgroups $\{H_\l\}_{\l\in\L}$ be hyperbolically embedded in $(G,X)$. 
 Then for each $n\geq 1$ there exists a copy $F$ of the free group on $n$ generators inside $G$ so that $\{H_\l\}_{\l\in\L}\cup\{F\times K(G)\}$ is hyperbolically embedded in $(G,X)$.
\end{thm}

The two improvements on \cite[Theorem 6.14]{DGO} that we make are that
\begin{enumerate}
 \item \cite[Theorem 6.14]{DGO} guarantees the existence of a virtually free hyperbolically embedded subgroup, but it does not allow to keep the $H_\l$'s in the hyperbolically embedded family, and
 \item \cite[Theorem 6.14]{DGO} does not explicitly describe some $Y\subseteq G$ so that the virtually free subgroup is hyperbolically embedded in $(G,Y)$.
\end{enumerate}

The improvement that we actually need is the first one. However, it is worth pointing out that keeping track of the generating set is also part of \cite[Theorem 3.15]{Hull-smallcanc} and \cite[Theorem 3.9]{AMS-outg}, and in those papers such control turns out to be useful.

Along the way, we will show (a slightly more general form of) the following fact that might be of independent interest. 
An analogous statement in the setting of relatively hyperbolic groups was shown in \cite{Os-elem}. 
As usual $\{H_\lambda\}_{\l\in\L}$ will be a non-degenerate hyperbolically embedded family of the group $G$, and we will denote by $\calH$ the set $\sqcup_{\l\in\L} (H_\lambda\backslash\{1\})$.
For a definition of quasi-convexity and geometric separation see Subsection~\ref{prelim:sub}.

\begin{prop}
 Let $\{H_\l\}_{\l\in\L}$ be hyperbolically embedded in $(G,X)$, for $X$ a generating system of the group $G$. 
 Suppose that $\{E_i\}_{i\in I}$ is a finite family of finitely generated subgroups of $G$ satisfying the following properties:
\begin{enumerate}
 \item Each $E_i$ is quasi-convex as a subset of $\Gamma=\Cay(G,X\sqcup\calH)$.
 \item The metric of $\Gamma$ restricted to $E_i$ is proper (i.e. balls of finite radius are finite).
 \item The family of all cosets of the $E_i$'s, regarded as a family of (labelled) subsets of $\Gamma$, is geometrically separated.
\end{enumerate}
Then $\{H_\l\}_{\l\in\L}\cup\{E_i\}_{i\in I}\hookrightarrow_h (G,X)$.
\end{prop}

\subsection{Preliminary facts}\label{prelim:sub}
In this subsection we collect a few facts that will be needed for the proof of Theorem \ref{dgo+}.

The following characterizations of hyperbolically embedded subgroups and relative hyperbolicity turn out to be convenient for the proof and, in the authors' opinion, they allow to provide a clear, worth-being-presented picture of why Theorem \ref{dgo+} is true.

The reader unfamiliar with asymptotic cones and ultralimits is referred to \cite{Dr-surv}. The following heuristic should however be enough to understand at least the ideas behind the proofs:
\begin{enumerate}
 \item an asymptotic cone of the metric space $(X,d)$ is a ``limit'', in some suitable sense, of rescaled copies $(X,d/n)$ of $X$.
 \item the ultralimit of the sequence of subsets $(A_i)$ of $X$ in an asymptotic cone of $X$ consists of all limit points of sequences $(x_i)$ with $x_i\in A_i$.
\end{enumerate}

\begin{defn}[\cite{DS-treegr}]\label{treegr:defn}
Let $Y$ be a geodesic metric space and $\calQ$ a collection of closed connected subsets of $Y$. Then $Y$ is \emph{tree-graded} with respect to $\calQ$ if

$(T_1)$ for any distinct $P,Q\in\calQ$ we have $|P\cap Q|\leq 1$, and

$(T_2)$ any simple geodesic triangle is contained in some $P\in\calQ$.

A simple geodesic triangle is a geodesic triangle with the property that distinct edges only intersect at their common endpoint.

Let $X$ be a geodesic metric space and $\calP$ a collection of subsets. Then $X$ is \emph{hyperbolic relative to }$\calP$ if every asymptotic cone of $X$ is tree-graded with respect to all (non-empty) ultralimits of sequences of elements of $\calP$, and two such ultralimits coincide and contain at least two points only if the corresponding sequences coincide almost everywhere with respect to the ultrafilter chosen to construct the asymptotic cone.
\end{defn}

(In \cite{DS-treegr} the latter condition is not
explicitly stated but ultralimits of sequences that are not
almost-everywhere equal are regarded as distinct sets throughout the paper.)

\begin{thm}[\cite{Si-metrrh}]
The family of subgroups $\{H_\l\}_{\l\in\L}$ is hyperbolically embedded in $(G,X)$, where $X$ is a generating system for $G$, 
if and only if the Cayley graph $\Cay(G,X)$ is hyperbolic relative to the collection of all cosets of the $H_\l$'s and the restriction of 
the metric of $\Cay(G,X)$ to each $H_\l$ is proper (i.e.~balls of finite radius are finite).
\end{thm}

Following \cite{DGO}, we say that the collection of subsets $\calA$ of a given metric space is \emph{geometrically separated} if for each $D\geq 0$ there exists $K_D\geq 0$ so that for any distinct $A,B\in\calA$ we have $\diam(N_D(A)\cap B)\leq K_D$.

The following variation of the definition of geometric separation will be useful:

\begin{defn}
 The collection of subsets $\calA$ of a given metric space is $K$-\emph{linearly geometrically separated} ($K$-LGS) if for any distinct $A,B\in\calA$ and any $D\geq 1$ we have ${\rm diam}(N_D(A)\cap B)\leq KD$.
\end{defn}

The following fact is a straightforward consequence of the definition.

\begin{lemma}
\label{t1fromlgs}
 If the collection of subsets $\calA$ of a metric space $X$ is $K$-LGS then in any asymptotic cone of $X$ distinct ultralimits of elements of $\calA$ intersect in at most one point.
\end{lemma}

\begin{proof}
 Suppose that the conclusion is not true. Then there exist an ultrafilter $\omega$, a sequence of scaling factors $(r_i)$, sequences of sets $(A^j_i)$ ($j=1,2$) with $A^1_i\neq A^2_i$ $\omega$-a.e. and sequences of points $(x_{k,i})$ ($k=1,2$) with the following properties:.
\begin{enumerate}
 \item $x_{k,i}\in A^1_i$.
 \item $\omega$-$\lim$ $d(x_{k,i},A^2_i)/r_i=0$ for $k=1,2$.
 \item $\omega$-$\lim$ $d(x_{1,i},x_{2,i})/r_i=\epsilon>0$.
\end{enumerate}

In particular, for $\omega$-a.e. $i$, we have $d(x_{1,i},x_{2,i})>\epsilon r_i/2$ and $x_{k,i}\in A^1_i\cap N_{\epsilon r_i/(2 K)} (A^2_i)$, in contradiction with the definition of $K$-LGS.
\end{proof}

Passing from geometric separation to linear geometric separation will be easy in our context due to the following (folklore) fact. We say that a subset $S$ of a hyperbolic metric space is \emph{quasi-convex} if there exists $\sigma\geq 0$ so that all geodesics connecting pairs of points on $S$ are contained in the $\sigma$-neighborhood of $S$. A family of subsets is \emph{uniformly quasi-convex} if all subsets in the family are quasi-convex with the same constant $\sigma$.

\begin{lemma}
\label{lgs}
 Suppose that $\calA$ is a collection of uniformly quasi-convex subsets of the hyperbolic metric space $X$. Then $\calA$ is geometrically separated if and only if it is $K$-LGS for some $K\geq 0$.
\end{lemma}

\begin{proof}
Clearly, we only need to show that geometric separation implies linear geometric separation.

 Suppose that $X$ is $\delta$-hyperbolic, that every $A\in\calA$ is $\sigma$-quasi-convex and that the diameter of the intersection of the $(\sigma+2\delta)$-neighborhoods of distinct elements of $\calA$ is a most $\kappa$. Fix any $D\geq 1$ and suppose by contradiction that there exist $x,y\in N_D(A)\cap B$ with $d(x,y)\geq 2(D+2\delta+2)+\kappa+1$ for some distinct $A,B\in\calA$. Let $x_1,y_1\in A$ be so that $d(x,x_1),d(y,y_1)\leq D+1$. For each $a,b\in X$, denote by $[a,b]$ any choice of geodesic between $a$ and $b$. Consider the point $p$ (resp. $q$) along $[x,y]$ at distance $D+2\delta+2$ from $x$ (resp. $y$). Notice that $d(p,q)\geq \kappa+1$. We claim that $p,q\in N_{\sigma}(B)\cap N_{\sigma+2\delta}(A)$, in contradiction with the definition of $\kappa$.
 
 The fact that $p,q$ are within distance $\sigma$ from $B$ is just a consequence of quasi-convexity. 
 Notice that $p$ is within distance $2\delta$ from $[x,x_1]\cup[x_1,y_1]\cup [y,y_1]$. 
 We have
 $$
 d(p,[x,x_1]\cup [y,y_1])\geq \min \{d(p,x),d(p,y)\}-\max \{d(x,x_1),d(y,y_1)\}>2\delta
 $$
 so $p$ must be within distance $2\delta$ from $[x_1,y_1]$, which in turn is contained in the $\sigma$-neighborhood of $A$
 (see Figure~\ref{linearsep}). 
 A similar argument also works for $q$, and this completes the proof.
\end{proof}

\begin{figure}[h]
 \includegraphics[width=10cm]{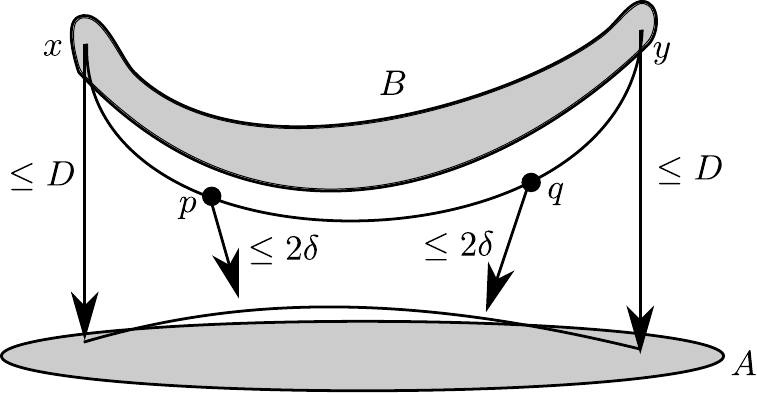}
\caption{}
\label{linearsep}
\end{figure}

We can now prove the following:
\begin{prop}
\label{addtransv}
 Let $\{H_\l\}_{\l\in\L}$ be hyperbolically embedded in $(G,X)$, for $X$ a generating system of the group $G$, and suppose that $\mathcal E\subseteq G$ is such that $\Gamma=\Cay(G,X\sqcup \calH\sqcup \mathcal E)$ is hyperbolic (e.g. $\mathcal E=\emptyset$).
 Suppose that $\{E_i\}_{i\in I}$ is a finite family of finitely generated subgroups of $G$ satisfying the following properties:
\begin{enumerate}
 \item Each $E_i$ is quasi-convex as a subset of $\Gamma$.
 \item The metric of $\Gamma$ restricted to $E_i$ is proper.
 \item The family of all cosets of the $E_i$'s, regarded as a family of subsets of $\Gamma$, is geometrically separated.
\end{enumerate}
Then $\{H_\l\}_{\l\in\L}\cup\{E_i\}_{i\in I}\hookrightarrow_h (G,X)$.
\end{prop}

\begin{proof}
Let us denote by $d_\Gamma$ the obvious metric on $\Gamma$.
First of all, we remark that we can replace $(2)$ with the condition

$(2')$ The metric of $\Gamma$ restricted to $E_i$ is quasi-isometric to a word metric on $E_i$.

In fact, thanks to quasi-convexity we can choose $D>0$ such that the $D$-neighborhood $N_i$ of $E_i$ in $\Gamma$ contains all the geodesics
joining points of $E_i$. Then, 
the inclusion of $E_i$ (endowed with the restriction of $d_\Gamma$) into 
$N_i$ (endowed with the path metric induced from $d_\Gamma$) is an isometric embedding. As $E_i$ acts properly and coboundedly on the path metric space $N_i$, 
the Milnor-Svarc Lemma tells us that the inclusion of $E_i$ (now endowed with a word metric)
into $N_i$ is a quasi-isometry, so that the desired conclusion follows.

We need to show that, given an asymptotic cone of $\Cay(G,X)$, the ultralimits of cosets of the $H_\l$'s and the $E_i$'s are connected and 
they satisfy properties $(T_1)$ and $(T_2)$ of Definition~\ref{treegr:defn} (ultralimits of sequences of sets are always closed).

First of all, observe that property $(T_2)$ holds just because it holds for the ultralimits of the $H_\l$'s already.

Ultralimits of the $E_i$'s are connected because they are bi-Lipschitz copies of an asymptotic cone of one of the $E_i$'s: this is a consequence of $(2')$ and the fact that the inclusion of $\Cay(G,X)$ in $\Gamma$ is $1$-Lipschitz, which together imply that the restriction of the metric of $\Cay(G,X)$ to $E_i$ is quasi-isometric to a word metric. 

Also, conditions $(3)$, $(1)$ and Lemma \ref{lgs} imply that there exists $K$ so that the family $\calB$ of the cosets of the $E_i$'s is $K$-LGS in $\Gamma$. 
For any subset $B\subseteq G$, let us denote by $N_D(B)$ (resp.~$N_D^\Gamma(B)$) the $D$-neighborhood
of $B$ in $G$ with respect to the word metric of $\Cay(G,X)$ (resp.~the metric $d_\Gamma$).
Now, if for some distinct $B_1, B_2\in \calB$ and $D\geq 1$ we have $x,y\in B_1\cap N_D(B_2)$,  then we also have $x,y\in B_1\cap N^\Gamma_D(B_2)$ because the inclusion of $\Cay(G,X)$ in $\Gamma$ is $1$-Lipschitz. But then we get $d_\Gamma(x,y)\leq KD$. 
As this holds for any such pair $x,y$, in view of $(2')$ it is easy to deduce that $\calB$ is $K'$-LGS in $\Cay(G,X)$ 
for some suitable $K'$ (as
$$d_\Gamma(x,y)\leq d_{\Cay(G,X)}(x,y)\leq C_1 d^w_{B_1}(x,y)\leq C_2 d_\Gamma(x,y),$$
where $d^w_{B_1}$ is any word metric on $B_1$ and $C_i$ are suitable constants).

In particular, distinct ultralimits of cosets of the $E_i$'s intersect in at most one point by Lemma \ref{t1fromlgs}. 
It is also easy to check that an ultralimit of cosets of the $E_i$'s cannot intersect in more than one point an ultralimit of cosets of the $H_\l$: 
this is because a coset of an $H_\l$ has diameter 1 in $\Gamma$ so that, by $(2)$, linear geometric separation holds in $\Gamma$ and hence in $\Cay(G,X)$ as well.

So, keeping into account that property $(T_1)$ holds for ultralimits of cosets of the $H_\l$'s, we now showed that property $(T_1)$ holds for ultralimits of cosets of the $H_\l$'s and the $E_i$'s, and the proof is complete.
\end{proof}

\subsection{Proof of Theorem \ref{dgo+}} We can now show how to adapt the proof of \cite[Theorem 6.14]{DGO} to prove Theorem \ref{dgo+}, whose notation we fix from now on.

Proceeding as in \cite{DGO}, we consider the action of $G$ on $\Gamma=\Cay(G,X\sqcup \calH)$. 
By \cite[Lemma 6.17]{DGO}, there exist elements $h_1,h_2$ in $G$ so that

\begin{enumerate}
 \item each $h_i$ acts hyperbolically on $\Gamma$,
 \item each $h_i$ is contained in a maximal elementary subgroup $E(h_i)$,
 \item the family of the cosets of the $E(h_i)$'s is geometrically separated when such cosets are regarded as subsets of $\Gamma$,
 \item $E(h_1)\cap E(h_2)= K(G)$.
\end{enumerate}

We need to modify the next step in the proof from \cite{DGO} a bit. In fact, \cite[Lemma 6.18]{DGO} states that for any $n$ there exists $Y\subseteq G$ 
and elementary subgroups $\{E_i\}_{i=1,\dots,n}$ of $G$ such that $\{E_i\}_{i=1,\ldots,n}\hookrightarrow_h (G,Y)$ and 
each $E_i$ is of the form $\langle g_i\rangle\times K(G)$. We need to replace this result with the following:
\begin{lemma}
\label{618+}
For each integer $n\geq 1$ there exist elementary subgroups $\{E_i\}_{i=1,\dots,n}$ of $G$ such that
 $\{H_\l\}_{\l\in\L}\cup \{E_i\}_{i=1,\ldots,n}$ is hyperbolically embedded in $(G,X)$ and each $E_i$ is of the form $\langle g_i\rangle\times K(G)$.
\end{lemma}

\begin{proof}
The proof of Lemma 6.18 in \cite{DGO} starts with choosing $Y\subseteq G$ so that $\{E(h_1),E(h_2)\}$ is hyperbolically embedded in $(G,Y)$. Instead of this, we invoke Proposition \ref{addtransv} and get that $\{H_\l\}_{\l\in\L}\cup \{E(h_1),E(h_2)\}$ is hyperbolically embedded in $(G,X)$.

The rest of the proof of \cite[Lemma 6.18]{DGO} takes place in $\Cay(G,Y\sqcup \mathcal E)$, for $\mathcal E=(E(h_1)\cup E(h_2))\backslash \{1\}$. However, the arguments apply to $\Gamma=\Cay(G, X\sqcup \calH \sqcup \mathcal E)$ as well, so that we can find pairwise non-commensurable elements $g_i$ of $G$ acting hyperbolically on $\Gamma$, each contained in a maximal elementary subgroup $E(g_i)$ which has the form $\langle g_i \rangle\times K(G)$. We can then invoke again Proposition \ref{addtransv} to add the $E(g_i)$'s to the list of hyperbolically embedded subgroups.

Finally, the $E(h_i)$'s can be removed from the list, because the ultralimits of their cosets are bi-Lipschitz copies of $\R$, so they do not contain simple geodesic triangles and hence do not affect property $(T_2)$ (while property $(T_1)$ is preserved under 
passing to smaller collections of subsets).
\end{proof}

We can now use the argument after \cite[Lemma 6.18]{DGO} with $Y=X$ and $\mathcal E=\calH \cup \bigsqcup_{i=1}^n ( E_i\backslash\{1\})$ and construct a subgroup $H=F\times K(G)<G$, with $F$ free of rank $n$. The authors of \cite{DGO} then check that the hypotheses of \cite[Theorem 4.42]{DGO} hold for $F$ and the action of $G$ on $\Gamma_{\mathcal E}=\Cay(G,X\sqcup \mathcal E)$, namely they show that

\begin{enumerate}
 \item $F$ is quasi-convex as a subset of $\Gamma_{\mathcal E}$,
 \item $F$ acts properly on $\Gamma_{\mathcal E}$,
 \item the cosets of $F$ are geometrically separated as subsets of $\Gamma_{\mathcal E}$.
\end{enumerate}

Hence, Proposition \ref{addtransv} allows us to add $H$ to the list of hyperbolically embedded subgroups. \qed

\bibliography{ext}{}

\newcommand{\etalchar}[1]{$^{#1}$}
\def\cprime{$'$}\def\polhk\#1{\setbox0=\hbox{\#1}{??oalign{\hidewidth
  \lower1.5ex\hbox{`}\hidewidth\crcr\unhbox0}}}
\begin{thebibliography}{{Hum}12}

\bibitem[AMS13]{AMS-outg}
Y.~{Antolin}, A.~{Minasyan}, and A.~{Sisto}.
\newblock {Commensurating endomorphisms of acylindrically hyperbolic groups and
  applications}.
\newblock {\em ArXiv:\href{http://arxiv.org/abs/1310.8605}{\tt{1310.8605}}},
  October 2013.

\bibitem[BBF10]{BBF}
M.~{Bestvina}, K.~{Bromberg}, and K.~{Fujiwara}.
\newblock {Constructing group actions on quasi-trees and applications to
  mapping class groups}.
\newblock {\em ArXiv:\href{http://arxiv.org/abs/1006.1939}{\tt{1006.1939}}},
  June 2010.

\bibitem[BBF13]{BBF2}
M.~{Bestvina}, K.~{Bromberg}, and K.~{Fujiwara}.
\newblock {Bounded cohomology via quasi-trees}.
\newblock {\em ArXiv:\href{http://arxiv.org/abs/1306.1542}{\tt{1306.1542}}},
  June 2013.

\bibitem[BBF{\etalchar{+}}14]{BBFIPP}
M.~{Bucher}, M.~{Burger}, R.~{Frigerio}, A.~{Iozzi}, C.~{Pagliantini}, and
  M.~B. {Pozzetti}.
\newblock Isometric embeddings in bounded cohomology.
\newblock {\em J. Topol. Anal.}, 6:1--25, 2014.

\bibitem[BF02]{BeFu-wpd}
Mladen Bestvina and Koji Fujiwara.
\newblock {Bounded cohomology of subgroups of mapping class groups}.
\newblock {\em Geom. Topol.}, 6:69--89 (electronic), 2002.

\bibitem[Bro81]{Brooks}
R.~Brooks.
\newblock Some remarks on bounded cohomology.
\newblock In Princeton~Univ. Press, editor, {\em Riemann surfaces and related
  topics: {P}roceedings of the 1978 {S}tony {B}rook {C}onference ({S}tate
  {U}niv. {N}ew {Y}ork, {S}tony {B}rook, {N}.{Y}., 1978)}, volume~97 of {\em
  Ann. of Math. Stud.}, Princeton, N.J., 1981.

\bibitem[CFI]{CFI}
I.~Chatterji, T.~{Fern{\'o}s}, and A.~Iozzi.
\newblock The median class and superrigidity of actions on {CAT}(0) cube
  complexes.
\newblock to appear in J. Topol., available at
  ArXiv:\href{http://arxiv.org/abs/1212.1585}{\tt{1212.1585}}.

\bibitem[DGO11]{DGO}
F.~{Dahmani}, V.~{Guirardel}, and D.~{Osin}.
\newblock Hyperbolically embedded subgroups and rotating families in groups
  acting on hyperbolic spaces.
\newblock to appear in Mem. Amer. Math. Soc., available at
  ArXiv:\href{http://arxiv.org/abs/1111.7048}{\tt{1111.7048}}, nov. 2011.

\bibitem[Dru02]{Dr-surv}
C.~Dru{\c{t}}u.
\newblock {Quasi-isometry invariants and asymptotic cones}.
\newblock {\em Internat. J. Algebra Comput.}, 12(1-2):99--135, 2002.
\newblock International Conference on Geometric and Combinatorial Methods in
  Group Theory and Semigroup Theory (Lincoln, NE, 2000).

\bibitem[DS05]{DS-treegr}
C.~Dru{\c{t}}u and M.~Sapir.
\newblock {Tree-graded spaces and asymptotic cones of groups}.
\newblock {\em Topology}, 44(5):959--1058, 2005.
\newblock With an appendix by D. Osin and Sapir.

\bibitem[EF97]{EpsteinFuji}
D.~B.~A. Epstein and K.~Fujiwara.
\newblock The second bounded cohomology of word-hyperbolic groups.
\newblock {\em Topology}, 36:1275--1289, 1997.

\bibitem[Far98]{Fa-relhyp}
B.~Farb.
\newblock Relatively hyperbolic groups.
\newblock {\em Geom. Funct. Anal.}, 8:810--840, 1998.

\bibitem[FM11]{Fuji_Man}
K.~Fujiwara and J.~F. Manning.
\newblock Simplicial volume and fillings of hyperbolic manifolds.
\newblock {\em Algebr. Geom. Topol.}, 11:2237--2264, 2011.

\bibitem[Fuj98]{Fujiwara1}
K.~Fujiwara.
\newblock The second bounded cohomology of a group acting on a
  {G}romov-hyperbolic space.
\newblock {\em Proc. London Math. Soc.}, 76:70--94, 1998.

\bibitem[Fuj00]{FujiTAMS}
K.~Fujiwara.
\newblock The second bounded cohomology of an amalgamated free product of
  groups.
\newblock {\em Trans. Amer. Math. Soc.}, 352:1113--1129, 2000.

\bibitem[Ham08]{Ha-isomhyp}
U.~Hamenst\"adt.
\newblock {Bounded cohomology and isometry groups of hyperbolic spaces}.
\newblock {\em J. Eur. Math. Soc. (JEMS)}, 10(2):315--349, 2008.

\bibitem[HO13]{HullOsin}
M.~Hull and D.~Osin.
\newblock Induced quasicocycles on groups with hyperbolically embedded
  subgroups.
\newblock {\em Algebr. Geom. Topol.}, 13:2635--2665, 2013.

\bibitem[{Hul}13]{Hull-smallcanc}
M.~{Hull}.
\newblock {Small cancellation in acylindrically hyperbolic groups}.
\newblock {\em ArXiv:\href{http://arxiv.org/abs/1308.4345}{\tt{1308.4345}}},
  August 2013.

\bibitem[{Hum}12]{Hu-bbf-qtreegr}
D.~{Hume}.
\newblock {Embedding mapping class groups into finite products of trees}.
\newblock {\em ArXiv:\href{http://arxiv.org/abs/1207.2132}{\tt{1207.2132}}},
  July 2012.

\bibitem[LR01]{Reid}
D.~D. Long and A.~W. Reid.
\newblock Constructing hyperbolic manifolds which bound geometrically.
\newblock {\em Math. Res. Lett.}, 8:443--455, 2001.

\bibitem[Min01]{Min}
I.~Mineyev.
\newblock Straightening and bounded cohomology of hyperbolic groups.
\newblock {\em Geom. Funct. Anal.}, 11:807--839, 2001.

\bibitem[MM85]{Matsu-Mor}
S.~Matsumoto and S.~Morita.
\newblock Bounded cohomology of certain groups of homeomorphisms.
\newblock {\em Proc. Amer. Math. Soc.}, 94:539--544, 1985.

\bibitem[MMS04]{MMS}
I.~Mineyev, N.~Monod, and Y.~Shalom.
\newblock Ideal bicombings for hyperbolic groups and applications.
\newblock {\em Topology}, pages 1319--1344, 2004.

\bibitem[MO]{MO}
A.~{Minasyan} and D.~{Osin}.
\newblock Acylindrical hyperbolicity of groups acting on trees.
\newblock to appear in Math. Ann., available at
  ArXiv:\href{http://arxiv.org/abs/1310.6289}{\tt{1310.6289}}.

\bibitem[Mon06]{Monod}
N.~Monod.
\newblock An invitation to bounded cohomology.
\newblock In Z\"urich Eur. Math.~Soc., editor, {\em International Congress of
  Mathematicians. Vol. II}, pages 1183--1211, 2006.

\bibitem[MRS13]{MRS-collisions}
D.~B. Mc{R}eynolds, Alan~W. Reid, and Matthew Stover.
\newblock Collisions at infinity in hyperbolic manifolds.
\newblock {\em Math. Proc. Cambridge Philos. Soc.}, 155(3):459--463, 2013.

\bibitem[MS06]{MS3}
N.~Monod and Y.~Shalom.
\newblock Orbit equivalence rigidity and bounded cohomology.
\newblock {\em Ann. of Math. (2)}, 164(3):825--878, 2006.

\bibitem[MS13]{MS-prodtrees}
J.~M. {Mackay} and A.~{Sisto}.
\newblock Embedding relatively hyperbolic groups in products of trees.
\newblock {\em Algebr. Geom. Topol.}, 13:2261--2282, 2013.

\bibitem[NM03]{MS1}
Y.~Shalom N.~Monod.
\newblock Negative curvature from a cohomological viewpoint and cocycle
  superrigidity.
\newblock {\em C. R. Math. Acad. Sci. Paris}, (10):635--638, 2003.

\bibitem[NM04]{MS2}
Y.~Shalom N.~Monod.
\newblock Cocycle superrigidity and bounded cohomology for negatively curved
  spaces.
\newblock {\em J. Differential Geom.}, 67:395--455, 2004.

\bibitem[Osi]{Os-acyl}
D.~Osin.
\newblock Acylindrically hyperbolic groups.
\newblock to appear in Trans. Amer. Math. Soc., available at
  ArXiv:\href{http://arxiv.org/abs/1304.1246}{\tt{1304.1246}}.

\bibitem[Osi06]{Os-elem}
Denis~V. Osin.
\newblock {Elementary subgroups of relatively hyperbolic groups and bounded
  generation}.
\newblock {\em Internat. J. Algebra Comput.}, 16(1):99--118, 2006.

\bibitem[{Sis}11]{Si-contr}
A.~{Sisto}.
\newblock {Contracting elements and random walks}.
\newblock {\em ArXiv:\href{http://arxiv.org/abs/1112.2666}{\tt{1112.2666}}},
  December 2011.

\bibitem[{Sis}12]{Si-metrrh}
A.~{Sisto}.
\newblock {On metric relative hyperbolicity}.
\newblock {\em ArXiv:\href{http://arxiv.org/abs/1210.8081}{\tt{1210.8081}}},
  October 2012.

\bibitem[Sis13]{Si-proj}
A.~Sisto.
\newblock {Projections and relative hyperbolicity}.
\newblock {\em Enseign. Math. (2)}, 59(1-2):165--181, 2013.

\bibitem[Som97]{Soma3}
T.~Soma.
\newblock Bounded cohomology and topologically tame kleinian groups.
\newblock {\em Duke Math. J.}, 88:357--370, 1997.

\end{thebibliography}
\bibliographystyle{alpha}

\end{document}